\documentclass[12pt]{article}
\usepackage[top=2.54cm, bottom=2.54cm, left=2.2cm, right=2.2cm]{geometry}

\usepackage{upgreek,amsmath, amssymb,amsfonts}
\usepackage{bm,bbm}
\usepackage{enumerate,enumitem,multirow,diagbox,makecell}
\usepackage{hyperref,url,graphicx} 
\hypersetup{
	colorlinks,
	citecolor=blue,
	linkcolor=red,
	urlcolor=magenta}

\usepackage{cite} 
\usepackage{algorithm, algorithmic} 
\usepackage{relsize} 
\usepackage{cases} 
\usepackage{lipsum} 
\usepackage{caption} 
\usepackage{dblfloatfix} 

\usepackage{subfig}
\usepackage{graphicx}

\usepackage{setspace}

\usepackage{authblk}

\makeatletter
\g@addto@macro\normalsize{%
	\setlength\abovedisplayskip{1ex}
	\setlength\belowdisplayskip{1ex}
	\setlength\abovedisplayshortskip{1ex}
	\setlength\belowdisplayshortskip{1ex}
}
\makeatother

\usepackage{amsthm} 
\newtheorem{theorem}{Theorem}

\newtheorem{lemma}{Lemma}
\newtheorem{remark}{Remark}

\renewcommand{\eta}{\upeta}
\renewcommand{\alpha}{\upalpha}
\renewcommand{\beta}{\upbeta}
\renewcommand{\delta}{\updelta}
\renewcommand{\epsilon}{\upepsilon}
\renewcommand{\gamma}{\upgamma}
\renewcommand{\kappa}{\upkappa}
\renewcommand{\lambda}{\uplambda}
\renewcommand{\mu}{\upmu}
\renewcommand{\sigma}{\upsigma}
\renewcommand{\theta}{\uptheta}
\renewcommand{\varepsilon}{\upvarepsilon}
\renewcommand{\tau}{\uptau}
\renewcommand{\zeta}{\upzeta}
\renewcommand{\rho}{\uprho}
\renewcommand{\xi}{\upxi}

\newcommand{\bd}{\bm{d}}
\newcommand{\be}{\bm{e}}
\newcommand{\bg}{\bm{g}}
\newcommand{\bI}{\bm{I}}

\newcommand{\bu}{\bm{u}}
\newcommand{\bv}{\bm{v}}
\newcommand{\bw}{\bm{w}}
\newcommand{\bx}{\bm{x}}
\newcommand{\by}{\bm{y}}
\newcommand{\bz}{\bm{z}}

\newcommand{\bA}{\bm{A}}
\newcommand{\bB}{\bm{B}}

\newcommand{\bG}{\bm{G}}
\newcommand{\bH}{\bm{H}}
\newcommand{\bL}{\bm{L}}

\newcommand{\bP}{\bm{P}}
\newcommand{\bQ}{\bm{Q}}
\newcommand{\bS}{\bm{S}}

\newcommand{\bY}{\bm{Y}}

\newcommand{\btheta}{\bm{\theta}}

\newcommand{\bDelta}{\bm{\Delta}}
\newcommand{\bLambda}{\bm{\Lambda}}

\newcommand{\bzero}{\bm{0}}

\newcommand{\hbH}{\hat{\bm{H}}}
\newcommand{\hbtheta}{\hat{\bm{\theta}}}

\newcommand{\obH}{\overline{\bm{H}}}
\newcommand{\oobH}{\overline{\overline{\bm{H}}}}
\newcommand{\obB}{\overline{\bm{B}}}

\newcommand{\obLambda}{\overline{\bLambda}}

\newcommand{\tbu}{\tilde{\bm{u}}}
\newcommand{\tbv}{\tilde{\bm{v}}}
\newcommand{\tbA}{\tilde{\bm{A}}}

\newcommand{\tbDelta}{\tilde{\bm{\Delta}}}

\newcommand{\real}{\mathbb{R}}

\newcommand{\transpose}{\mathsmaller{T}}
\newcommand{\set}[1]{\left\{ #1 \right\}}

\newcommand{\parenthese}[1]{\left( #1 \right)}
\newcommand{\abs}[1]{\left| #1 \right|}
\newcommand{\norm}[1]{\| #1 \|}

\newcommand{\remove}[1]{}
\usepackage{color}

\usepackage{tikz}
\usetikzlibrary{shapes,arrows,chains}
\usetikzlibrary[calc]

\begin{document}
\textcopyright~2019 IEEE.  Personal use of this material is permitted.  Permission from IEEE must be obtained for all other uses, in any current or future media, including reprinting/republishing this material for advertising or promotional purposes, creating new collective works, for resale or redistribution to servers or lists, or reuse of any copyrighted component of this work in other works.
\thispagestyle{empty}
\newpage

\setcounter{page}{1}

\title{Efficient Implementation of Second-Order Stochastic Approximation Algorithms in High-Dimensional Problems}

\author[1]{Jingyi Zhu}
\author[1]{Long Wang}
\author[1,2]{James C. Spall \thanks{Correspondence should be addressed to James C. Spall: \href{mailto:james.spall@jhuapl.edu}{james.spall@jhuapl.edu}}}
\affil[1]{Department of Applied Mathematics and Statistics, Johns Hopkins University}
\affil[2]{Applied Physics Laboratory, Johns Hopkins University}

\date{}
 
%


\maketitle
\onehalfspacing

\begin{abstract}
	Stochastic approximation (SA) algorithms have been widely applied in minimization problems when the loss functions and/or the gradient information are only accessible through noisy evaluations. Stochastic gradient (SG) descent---a first-order algorithm and a workhorse of much machine learning---is perhaps the most famous form of SA. Among all SA algorithms, the second-order simultaneous perturbation stochastic approximation (2SPSA) and the second-order stochastic gradient (2SG) are particularly efficient in handling high-dimensional problems, covering both gradient-free and gradient-based scenarios. However, due to the necessary matrix operations, the per-iteration floating-point-operations (FLOPs) cost of the standard 2SPSA/2SG is $O(p^3)$, where $p$ is the dimension of the underlying parameter. Note that the $O(p^3)$ FLOPs cost is distinct from the classical SPSA-based per-iteration $O(1)$ cost in terms of the number of noisy function evaluations. In this work, we propose a technique to efficiently implement the 2SPSA/2SG algorithms via the symmetric indefinite matrix factorization and show that the FLOPs cost is reduced from $O(p^3)$ to $O(p^2)$. The formal almost sure convergence and rate of convergence for the newly proposed approach are directly inherited from the standard 2SPSA/2SG. The improvement in efficiency and numerical stability is demonstrated in two numerical studies.
	
\end{abstract}


%

\section{Introduction}

\subsection{Problem Context}
Stochastic approximation (SA) has been widely applied in minimization and/or root-finding problems, when only \emph{noisy} loss function and/or gradient evaluations are accessible. Consider minimizing a differentiable loss function $ L(\btheta): \real^p \to \real $, $p\ge 1$ being the dimension of $\btheta$, where only noisy evaluations of $L\parenthese{\cdot}$ and/or its gradient $\bg\parenthese{\cdot} $ are accessible. The key distinction between SA and classical deterministic optimization is the presence of noise, which is largely inevitable when the function measurements are collected from either physical experiments or computer simulation. Besides, the noise term comes into play when the loss function is only evaluated on a small subset of an entire (inaccessible) dataset as in online training methods popular with neural network and machine learning. In the era of big-data, we deal with applications where the solution is data dependent such that the cost is minimized over a given set of sampled data rather than the entire distribution. Overall, SA algorithms have numerous applications in adaptive control, natural language processing, facial recognition and collaborative filtering, to name a few.

In modern machine learning, there is a growing need for algorithms to handle high-dimensional problems. Particularly for deep learning, the need arises as the number of parameters (including both weights and bias) explodes quickly as the network depth and width increase. First-order methods based on back-propagation are widely applied, yet they suffer from slow convergence rate in later iterations after a sharp decline during early iterations. Second-order methods are occasionally utilized to speed up convergence in terms of the number of iterations, but, still, at a computational burden of $O(p^3)$ per-iteration floating-point-operations (FLOPs). 

The adaptive second-order methods here differ in fundamental ways from stochastic quasi-Newton and other similar methods in the machine learning literature. First, most of the machine learning-based methods are designed for loss functions of the empirical risk function (ERF) form, namely for functions represented as summations, where each summand represents the contribution of one data vector. Such a structure, together with an assumption of strong convexity, has been exploited in \cite{johnson2013accelerating,martens2015optimizing}. Second, first- or second-order derivative information is often assumed to be directly available on the summands in the loss function (e.g., \cite{byrd2016stochastic, sohl2014fast, schraudolph2007stochastic}). Ref. \cite{saab2019multidimensional} also assumes direct information on the Hessian is available in a second-order stochastic method, but allows for loss functions more general than the ERF. Ref. \cite{byrd2016stochastic} applies the BFGS method to stochastic optimization, but under a nonstandard setup where noisy Hessian information can be gathered. In our work, we assume only noisy loss function evaluations or noisy gradient information are available. Third, notions of convergence and rates of convergence are in line with those in deterministic optimization when the loss function (the ERF) is composed of a finite (although possibly large) number of summands. For example, in \cite{johnson2013accelerating, martens2015optimizing, byrd2016stochastic, sohl2014fast, schraudolph2007stochastic}, rates of convergence are linear or quadratic as a measure of iteration-to-iteration improvement in the ERF. In contrast, we follow the traditional notion of stochastic approximation, including applicability to general noisy loss functions, no availability of direct derivative information, and stochastic notions of convergence and rates of convergence based on sample-points (almost surely, a.s.) and convergence in distribution.

To achieve a faster convergence rate at a reasonable computational cost, we present a second-order simultaneous perturbation (SP) method that incurs only $O(p^2)$ per-iteration FLOPs, in contrast to the standard $O(p^3)$. The idea of SP is an elegant generalization of the finite difference (FD) scheme and can be applied in both first-order and second-order SA algorithms. Our proposed method rests on factorization of symmetric indefinite matrices. 

\subsection{Summary of SP-Based Methods}
Among various stochastic optimization schemes, SP algorithms are particularly efficient compared with the FD methods. Under certain regularity conditions, \cite{spall1992multivariate} shows that the simultaneous perturbation stochastic approximation (SPSA) algorithm uses only $ 1/p $ of the required number of loss function observations needed in the FD form to achieve the same level of mean-squared-error (MSE) for the SA iterates. To further explore the potential of SP algorithms, \cite{spall2000adaptive} presents the second-order SP-based methods, including second-order SPSA (2SPSA) for applications in the gradient-free case and the second-order stochastic gradient (2SG) for applications in the gradient-based case. Those methods estimate the Hessian matrix to achieve near-optimal or optimal convergence rates and can be viewed as the stochastic analogs of the deterministic Newton-Raphson algorithm. Ref. \cite{spall2009feedback} incorporates both a feedback process and an optimal weighting mechanism in the averaging of the per-iteration Hessian estimates to improve the accuracy of the cumulative Hessian estimate in E2SPSA (enhanced 2SPSA) and E2SG (enhanced 2SG). The guidelines for practical implementation details and the choice of gain coefficients are available in \cite{spall1998implementation}. SPSA is also capable in dealing with discrete variables as shown in \cite{wang2011discrete, wang2018mixed}. More details on the related methods are discussed in \cite[Chaps. 7\textendash 8]{bhatnagar2012stochastic}.

\subsection{Our Contribution}
Refs. \cite{spall2000adaptive,spall2009feedback} show that the 2SPSA/2SG methods can achieve near-optimal or optimal convergence rates with a much smaller number (independent of dimension $p$) of loss or gradient function evaluations relative to other second-order stochastic methods in \cite{fabian1971stochastic, ruppert1985newton}. However, after obtaining the function evaluations, the per-iteration FLOPs to update the estimate is $ O(p^3) $, as discussed below. The computational burden becomes more severe as $p$ gets larger. This is usually the case in many modern machine learning applications. Here we propose a scheme to implement 2SPSA/2SG efficiently via the symmetric indefinite factorization, which reduces the per-iteration FLOPs from $ O(p^3) $ to $ O(p^2) $. We also show that the proposed scheme inherits the almost sure convergence and the rate of convergence from the original 2SPSA/2SG in \cite{spall2000adaptive}.

The remainder of the paper is as follows. Section~\ref{sec:2SPSA} reviews the original 2SPSA/2SG in \cite{spall2000adaptive} along with the computational complexity analysis. Section~\ref{sec:efficient_implementation} discusses the proposed efficient implementation. Section~\ref{sec:theory} covers the almost sure convergence and asymptotic normality. Section~\ref{sec:discussion} discusses some practical issues. Numerical studies and conclusions are in Section~\ref{sec:numerical} and Section~\ref{sec:conclusion}, respectively.

\section{Review of 2SPSA/2SG}\label{sec:2SPSA}
Before proceeding, let us review the original 2SPSA/2SG algorithms and explain their $ O(p^3) $ per-iteration FLOPs.

\subsection{2SPSA/2SG Algorithm}
Following the standard SA framework, we find the root(s) of $\bg\parenthese{\btheta}\equiv \partial L\parenthese{\btheta}/\partial\btheta$ in order to solve the problem of finding $ \arg\min L\parenthese{\btheta} $. Our central task is to streamline the computing procedure, so we do not dwell on differentiating the global minimizer(s) from the local ones. Such a root-finding formulation is widely used in the neural network training and other machine learning literature. We consider optimization under two different settings:
\begin{enumerate}
	\item Only noisy measurements of the loss function, denoted by $ y(\btheta), $ are available.
	\item Only noisy measurements of the gradient function, denoted by $ \bY(\btheta) $, are available. 
\end{enumerate}
The conditions for noise can be found in \cite[Assumptions C.0 and C.2]{spall2000adaptive}, which include various type of noise such as Gaussian, multiplicative and impulsive noise as special cases.
The main updating recursion for 2SPSA/2SG in \cite{spall2000adaptive} is 
\begin{equation} \label{eq:theta_update}
\hbtheta_{k+1} = \hbtheta_k - a_k \oobH_k^{-1} \bG_k(\hbtheta_k), k = 0, 1, \dots ,
\end{equation}
where $ \{a_k\}_{k\geq0} $ is a positive decaying scalar gain sequence, $ \bG_k(\hbtheta_k) $ is the direct noisy observation or the approximation of the gradient information, and $ \oobH_k $ is the approximation of the Hessian information. The true gradient $ \bg(\hbtheta_k) $ is estimated by
\begin{numcases}
{\bG_k(\hbtheta_k)=}
\frac{y(\hbtheta_k+c_k\bDelta_k)-y(\hbtheta_k-c_k\bDelta_k)}{2c_k\bDelta_k}, &\hspace{-.25in}\text{for 2SPSA,}\label{eq:gradient_estimate_2SPSA}\\
\bY_k(\hbtheta_k), &\hspace{-.25in}\text{for 2SG,}\label{eq:gradient_estimate_2SG}
\end{numcases}
where $ \bDelta_k = [\Delta_{k1}, \dots, \Delta_{kp}]^\transpose $ is a mean-zero $p$-dimensional stochastic perturbation vector with bounded inverse moments \cite[Assumption B.6$^{\prime\prime}$ on pp. 183]{spall2005introduction}, $ 1 / \bDelta_k = \bDelta_k^{-1} \equiv [\Delta_{k1}^{-1}, \dots, \Delta_{kp}^{-1}]^\transpose $ is a vector of reciprocals of each nonzero components of $ \bDelta_k $ ($\bDelta_k^{-\transpose}$ is the transpose of $\bDelta_k^{-1}$), and $ \{c_k\}_{k\geq0} $ is a positive decaying scalar gain sequence satisfying conditions in \cite[Sect. 7.3]{spall2005introduction}. A valid choice for $ c_k $ is $ c_k = 1 / (k+1)^{1/6}$. For the Hessian estimate $ \oobH_k $, \cite{spall2000adaptive} proposes
\begin{numcases}{}
\oobH_k = \bm{f}_k(\obH_k), & \label{eq:H_ooverline}\\
\obH_k = (1-w_k) \obH_{k-1} + w_k \hbH_k ,& \label{eq:H_overline}\\
\hbH_k=\frac{1}{2}\left[\frac{\updelta\bG_k}{2c_k}\bDelta_k^{-\transpose}+\left(\frac{\updelta\bG_k}{2c_k}\bDelta_k^{-\transpose}\right)^\transpose \right] , \label{eq:H_hat}&\\
\updelta\bG_k=\bG_k^{(1)}(\hbtheta_k+c_k\bDelta_k)-\bG_k^{(1)}(\hbtheta_k-c_k\bDelta_k) , \nonumber&
\end{numcases}
where $ \bm{f}_k\hspace{-0.04in}: \real^{p\times p} \to \{$positive definite $p\times p$ matrices$\}$ is a preconditioning step to guarantee the positive-definiteness of $ \oobH_k $, $ \{w_k\}_{k\geq0} $ is a positive decaying scalar weight sequence, and $ \bG_k^{(1)}(\hbtheta_k\pm c_k\bDelta_k) $ are one-sided gradient estimates calculated by
\begin{align*}
&\bG_k^{(1)}(\hbtheta_k\pm c_k\bDelta_k) =
\begin{cases}
\frac{y(\hbtheta_k\pm c_k\bDelta_k+\tilde{c}_k\tbDelta_k)-y(\hbtheta_k\pm c_k\bDelta_k)}{\tilde{c}_k\tbDelta_k}, &\hspace{-.1in}\text{in 2SPSA,}\\
\bY_k(\hbtheta_k\pm c_k\bDelta_k), &\hspace{-.1in}\text{in 2SG,}
\end{cases}
\end{align*}
where $ \{\tilde{c}_k\}_{k\geq0} $ is another positive decaying gain sequence, and $ \tbDelta_k = [\tilde{\Delta}_{k1}, \dots, \tilde{\Delta}_{kp}]^\transpose $ is generated independently of $ \bDelta_k $, but in the same statistical manner as $ \bDelta_k $. Some valid choices for $w_k$ include $w_k=1/(k+1)$ and the asymptotically optimal choices in \cite[Eq. (4.2) or Eq. (4.3)]{spall2009feedback}. Ref. \cite{spall2000adaptive} considers the special case where $ w_k=1/\parenthese{k+1} $, i.e., $\obH_k$ is a sample average of the $ \hbH_j $ for $ j = 1, ..., k $. Later, \cite{spall2009feedback} proposes the E2SPSA and E2SG to obtain more accurate Hessian estimates by taking the optimal selection of weights and feedback-based terms in (\ref{eq:H_overline}) into account. While the focus of this paper is the original 2SPSA/2SG in \cite{spall2000adaptive}, we also discuss the applicability of the ideas to the E2SPSA/E2SG algorithms in \cite{spall2009feedback}. Note that, independent of $p$, one iteration of 2SPSA/E2SPSA uses four noisy measurements $ y(\cdot) $ and one iteration of 2SG/E2SG uses three noisy measurements $ \bY(\cdot) $.

\subsection{Per-Iteration Computational Cost of $ O(p^3) $} \label{subsect:p3cost}

The per-iteration computational cost of $ O(p^3) $ arises from two steps: one is from the preconditioning step in (\ref{eq:H_ooverline}), i.e., obtaining $ \oobH_k $; the other is from the descent direction step in (\ref{eq:theta_update}), i.e., obtaining $ \oobH_k^{-1}\bG_k(\hbtheta_k) $. We now discuss the per-iteration computational cost of these two steps in more detail.

\textbf{Preconditioning} The preconditioning step in (\ref{eq:H_ooverline}) is to guarantee the positive-definiteness of the Hessian estimate $ \oobH_k $. This step is necessary, because the updating of $ \obH_k $ in (\ref{eq:H_overline}) does not necessarily yield a positive-definite matrix (but $ \obH_k $ is guaranteed to be symmetric). One straightforward way is to perform the following transformation:
\begin{equation} \label{eq:f_k_sqrtm}
\bm{f}_k(\obH_k) = (\obH_k \obH_k + \delta_k \bI)^{1/2} ,
\end{equation}
where $ \delta_k > 0 $ is a small \emph{decaying} scalar coefficient \cite{spall2000adaptive} and superscript ``1/2" denotes the symmetric matrix square root. Let $ \lambda_i(\cdot) $ denote the $i$-th eigenvalue of the argument. Since $ \lambda_i(\bA+c\bI) = \lambda_i(\bA)+c $ for any matrix $\bA$ and constant $c$ \cite[Obs. 1.1.7]{horn1990matrix}, we see that (\ref{eq:f_k_sqrtm}) directly modifies the eigenvalues of $ \obH_k\obH_k $ such that $ \lambda_i(\obH_k\obH_k + \delta_k\bI) = \lambda_i(\obH_k\obH_k) + \delta_k $ for $ i = 1, ..., p $. When $ \delta_k > 0 $, all the eigenvalues of $ \obH_k\obH_k + \delta_k\bI $ are strictly positive and therefore the resulting $ \oobH_k $ is positive definite. However, (\ref{eq:f_k_sqrtm}) has a computational cost of $ O(p^3) $ due to both the matrix multiplication in $ \obH_k \obH_k $ and the matrix square root computing \cite{higham1987computing}. Another intuitive transformation is 
\begin{equation} \label{eq:f_k_add}
\bm{f}_k(\obH_k) = \obH_k + \delta_k \bI
\end{equation} for a positive and sufficiently large $ \delta_k $. Again, applying eigen-decomposition on $ \obH_k $, we see that $ \lambda_i(\oobH_k) = \lambda_i(\obH_k) + \delta_k $ for $ i = 1, \cdots, p $. Take $ \uplambda_{\min}\parenthese{\cdot}= \min_{1\le i\le p } \uplambda_i\parenthese{\cdot} $ for any argument matrix in $\real^{p\times p}$. Any $ \delta_k > |\lambda_{\min}(\obH_k)| $ will result in $ \lambda_{\min}(\oobH_k) > 0 $, and therefore the output $ \oobH_k $ is positive definite. Unfortunately, (\ref{eq:f_k_add}) cannot avoid the $O(p^3)$ cost in estimating $ \lambda_{\min}(\obH_k) $. 

Besides the $O(p^3)$ cost in (\ref{eq:f_k_sqrtm}) and (\ref{eq:f_k_add}), the Hessian estimate $ \oobH_k $ may be ill-conditioned, leading to slow convergence. Ref. \cite{zhu2002modified} proposes to replace all negative eigenvalues of $ \obH_k $ with values proportional to its smallest positive eigenvalue. Such modification is shown to improve the convergence rate for problems with ill-conditioned Hessian and achieve smaller mean square errors for problems with better-conditioned Hessian compared with original 2SPSA \cite{zhu2002modified}. However, those benefits are gained at a price of computing the eigenvalues of $ \obH_k $, which still costs $ O(p^3) $.

\textbf{Descent direction} Another per-iteration computational cost of $ O(p^3) $ originates from computing the descent direction in (\ref{eq:theta_update}), which is typically computed by solving the linear system for $ \bd_k: \oobH_k\bd_k = \bG_k(\hbtheta_k) $. The estimate is updated recursive as following: 
\begin{equation}\label{eq:theta_update_s}
\hbtheta_{k+1} = \hbtheta_k - a_k\bd_k .
\end{equation} 
With the matrix left-division, it is possible to efficiently solve for $ \bd_k $. However, the computation costs of typical methods, such as $LU$ decomposition or singular value decomposition, are still dominated by $ O(p^3) $.

To speed up the original 2SPSA/2SG, \cite{rastogi2016efficient} proposes to rearrange (\ref{eq:H_overline}) and (\ref{eq:H_hat}) into the following two sequential rank-one modifications,
\begin{numcases}{}
\obH_k = t_k\obH_{k-1} + b_k\tbu_k\tbu_k^{\transpose} - b_k\tbv_k\tbv_k^{\transpose}, & \label{eq:two_rank_one_update}\\
\tbu_k = \sqrt{\frac{\norm{\bv_k}}{2\norm{\bu_k}}} \parenthese{\bu_k + \frac{\norm{\bu_k}}{\norm{\bv_k}}\bv_k} , & \label{eq:u_k_tilde}\\
\tbv_k = \sqrt{\frac{\norm{\bv_k}}{2\norm{\bu_k}}} \parenthese{\bu_k - \frac{\norm{\bu_k}}{\norm{\bv_k}}\bv_k} , &\label{eq:v_k_tilde}
\end{numcases}
where the scalar terms $ t_k $ and $ b_k $ (\ref{eq:two_rank_one_update}), and vectors $ \bu_k $ and $ \bv_k $ in (\ref{eq:u_k_tilde}) and (\ref{eq:v_k_tilde}) are listed in Table~\ref{table:u_k_v_k}. Applying the matrix inversion lemma \cite[pp. 513]{spall2005introduction}, \cite{rastogi2016efficient} shows that $ \obH_k^{-1} $ can be computed from $ \obH_{k-1}^{-1} $ with a cost of $ O(p^2) $. However, the positive-definiteness of $ \obH_k^{-1} $ is not guaranteed, and an additional eigenvalue modification step similar to either (\ref{eq:f_k_sqrtm}) or (\ref{eq:f_k_add}) is required. As discussed before, for any direct eigenvalue modifications, the computational cost of $ O(p^3) $ is inevitable due to the lacking knowledge about the eigenvalues of $ \obH_{k-1}^{-1} $.

In short, no prior works can fully streamline the entire second-order SP procedure with an $O(p^2)$ per-iteration FLOPs, which motivates the elegant procedure below. 

\begin{table*}[!htbp]
	\renewcommand{\arraystretch}{2}
	\caption{Expressions for terms in (\ref{eq:two_rank_one_update})--(\ref{eq:v_k_tilde}). See \cite[Sect. 7.8.2]{spall2005introduction} for detailed suggestions.}
	\label{table:u_k_v_k}
	\centering
	\begin{tabular}{|l|c|c|c|c|}
		\hline
		Algorithm & $ t_k $ & $ b_k $ & $ \bu_k $ & $ \bv_k $\\
		\hline\hline	
		2SPSA \cite{spall2000adaptive} & $ 1 - w_k $ & $ w_k \delta y_k / (4c_k\tilde{c}_k) $ & $ \tbDelta_k^{-1} $ & \multirow{4}{*}{$ \bDelta_k^{-1}$} \\
		\cline{1-4}
		E2SPSA \cite{spall2009feedback} & $ 1 $ & $ w_k[\updelta y_k / (2c_k\tilde{c}_k) - \bDelta_k^\transpose\obH_{k-1}\tbDelta_k]/2 $ & $ \tbDelta_k^{-1} $ & \\
		\cline{1-4}
		2SG \cite{spall2000adaptive} & $ 1 - w_k $ & $ w_k / (4c_k) $ & $ \delta\bG_k $ & \\
		\cline{1-4}
		E2SG \cite{spall2009feedback} & $ 1 $ & $ w_k / 2 $ & $ \delta\bG_k/(2c_k) - \obH_{k-1} \bDelta_k $ & \\
		\hline
	\end{tabular}
\end{table*}

\section{Efficient implementation of 2SPSA/2SG} 
\label{sec:efficient_implementation}
\subsection{Introduction}\label{subsect:Introduction}
With the motivation for proposing an efficient implementation scheme for 2SPSA/2SG laid out in Subsection~\ref{subsect:p3cost}, we now explain our methodology in more detail. Note that none of the prior attempts on 2SPSA/2SG methods can bypass the end-to-end computational cost of $O(p^3)$ per iteration in high-dimensional stochastic optimization problems. Therefore, we propose to replace $ \obH_k $ by its symmetric indefinite factorization, which enables us to implement the 2SPSA/2SG at a per-iteration computational cost of $O(p^2)$. Our work helps alleviate the notorious curse of dimensionality by achieving, to the best of our knowledge, the fastest possible second-order methods based on Hessian estimation. Also, note that the techniques in \cite{rastogi2016efficient} are no longer applicable because our scheme keeps track of the matrix factorization in lieu of the matrix itself, so we propose new algorithms to establish our claims.

To better illustrate our scheme and to be consistent with the original 2SPSA/2SG, we decompose our approach into the following three main steps and discuss the efficient implementation step by step.

\begin{enumerate}
	\item[i)] \label{item:ranktwomodification} \textbf{Two rank-one modifications}: Update the symmetric indefinite factorization of $ \obH_k $ by the two sequential rank-one modifications in (\ref{eq:two_rank_one_update}).
	\item[ii)] \label{item:preconditioning} \textbf{Preconditioning}: Obtain the symmetric factorization of a positive definite $ \oobH_k $ from the symmetric factorization of $ \obH_k $.
	\item[iii)] \label{item:descentdirection} \textbf{Descent direction}: Update $ \hbtheta_{k+1} $ by recursion (\ref{eq:theta_update_s}).
\end{enumerate}

Note that $ \obH_k $ is guaranteed to be symmetric by (\ref{eq:two_rank_one_update}) as long as $ \obH_0 $ is chosen symmetric. For the sake of comparison, we present the flow-charts of the original 2SPSA and that of our proposed scheme in Figure~\ref{fig:Flow_chart}, 
along with the per-iteration and per-step computational cost. The comparison of the flow-charts helps to put the extra move of indefinite factorization into perspective. 

{ 
	\tikzset{font=\footnotesize} 
	\tikzstyle{block} = [rectangle, draw, fill=blue!20, text centered, rounded corners, minimum height=2em]
	\tikzstyle{line} = [draw, -latex']
	\begin{figure*}[!hbtp]
		\centering
		\subfloat[Flow chart for the original 2SPSA/2SG]{
			\begin{tikzpicture}[auto]
			\node [block] (hbtheta) {$ \hbtheta_k $};
			\node [block, right = 1.8cm of hbtheta, text width=1.5cm] (bG) {$ \bG_k(\hbtheta_k)$,\\$\hbH_k $};
			\node [block, right = 1.2cm of bG] (obH) {$ \obH_k $};
			\node [block, right = 2.4cm of obH] (oobH) {$ \oobH_k $};
			\node [block, right = 2.8cm of oobH] (direction) {$ \bd_k $};
			\node [block, right = 1.6cm of direction] (hbthetanew) {$ \hbtheta_{k+1} $};
			\path [line] (hbtheta) -- node {2SPSA/2SG} node[below] {$ O(p^2) $} (bG);
			\path [line] (bG) -- node {(\ref{eq:H_overline})} node[below] {$ O(p^2) $} (obH);
			\path [line] (obH) -- node[align=center] {maintain positive- \\ [-0.4ex] definiteness (\ref{eq:H_ooverline})} node[below] {$ O(p^3) $} (oobH);
			\path [line] (oobH) -- node[align=center]{solve full-rank system \\ [-0.3ex] via back-division} node[below] {$ O(p^3) $} (direction);
			\path [line] (direction) -- node[align=center] {recursive \\ [-0.4ex] update (\ref{eq:theta_update})} node[below] { $ O(p) $} (hbthetanew);
			\end{tikzpicture}
			\label{fig:Original}}
		\hfil
		\subfloat[Flow chart for the proposed efficient implementation of 2SPSA/2SG (see Section~\ref{subsec:complexity} for detailed description)]{
			\begin{tikzpicture}[auto]
			\node [block] (hbtheta) {$ \hbtheta_k $};
			\node [block, right = 1.7cm of hbtheta, text width=1.5cm] (bG) {$ \bG_k(\hbtheta_k)$, \\ $\tbu_k, \tbv_k $};
			\node [block, right = 2.5cm of bG, text width=1.7cm] (obH) {factorization of $ \obH_k $};
			\node [block, right = 1.8cm of obH, text width=1.7cm] (oobH) {factorization of $ \oobH_k $};
			\node [block, right = 2cm of oobH] (direction) {$ \bd_k $};
			\node [block, right = 0.8cm of direction] (hbthetanew) {$ \hbtheta_{k+1} $};
			\path [line] (hbtheta) -- node {2SPSA/2SG} node[below] {$ O(p) $} (bG);
			\path [line] (bG) -- node[align=center] {two rank-one updates \\ [-0.4ex] via Algo.~\ref{algo:two_rank_one_update}} node[below] {$ O(p^2) $} (obH);
			\path [line] (obH) -- node[align=center] {preconditioning\\ [-0.4ex] via Algo.~\ref{algo:preconditioning}} node[below] {$ O(p^2) $} (oobH);
			\path [line] (oobH) -- node[align=center] {solve triangular systems\\ [-0.4ex] via Algo.~\ref{algo:descent_direction}} node[below] {$ O(p^2) $} (direction);
			\path [line] (direction) -- node {(\ref{eq:theta_update})} node[below] {$ O(p) $} (hbthetanew);
			\end{tikzpicture}
			\label{fig:Proposed}}
		\captionsetup{justification=centering}
		\caption{Flow charts showing FLOPs cost at each stage of the original 2SPSA/2SG and the proposed 2SPSA/2SG. Algorithms~\ref{algo:two_rank_one_update}--\ref{algo:descent_direction} in the lower path are described in Section~\ref{subsec:algo}.}
		\label{fig:Flow_chart}
	\end{figure*}
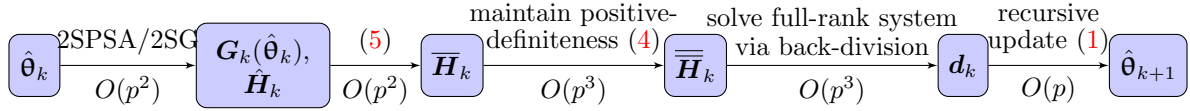
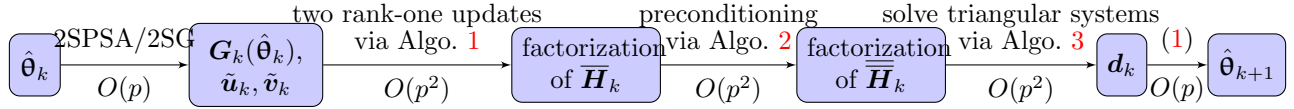
}

The remainder of this section is as follows. We introduce the symmetric indefinite factorization in Subsection~\ref{subsec:IMF} and derive the efficient algorithms in Subsection~\ref{subsec:algo}. The per-iteration computational complexity analysis is included in Subsection~\ref{subsec:complexity}. 

\subsection{Symmetric Indefinite Factorization} \label{subsec:IMF}

This subsection briefly reviews the symmetric indefinite factorization, also called $ \bL\bB\bL^\transpose $ factorization, introduced in \cite{bunch1971direct}, which applies to any symmetric matrix $ \obH $ regardless of the positive-definiteness:
\begin{equation}
\label{eq:LBL}
\bP \obH \bP^\transpose = \bL\bB\bL^\transpose, 
\end{equation}
where $ \bP $ is a permutation matrix, $ \bB $ is a block diagonal matrix with diagonal blocks being symmetric with size $ 1\times1 $ or $ 2\times2 $, and $ \bL $ is a lower-triangular matrix. Furthermore, the matrices $ \bL $ and $ \bB $ satisfy the following properties \cite[Sect. 4]{bunch1971direct}, which are fundamental for carrying out subsequent steps i) -- iii) at a computational cost of $ O(p^2) $:
\begin{itemize}
	\item The magnitudes of the entries of $ \bL $ are bounded by a fixed positive constant. Moreover, the diagonal entries of $ \bL$ are all equal to $1$.
	\item $ \bB $ has the same number of positive, negative, and zero eigenvalues as $ \obH $.
	\item The number of negative eigenvalues of $ \obH $ is the sum of the number of blocks of size $ 2\times2 $ on the diagonal and the number of blocks of size $ 1\times1 $ on the diagonal with negative entries of $ \bB $. (Note: There are no guarantees for the signs of the entries in the $ 2\times2 $ blocks.)
\end{itemize}

The bound on the magnitudes of the entries of $ \bL $ is approximately $ 2.7808 $ per \cite{bunch1977some} and it is \textit{independent} of the size of $ \obH $. As shown in Theorems~\ref{thm:H_barbar_property}--\ref{thm:H_barbar_uniform_bound}, such a constant bound is useful in practice to perform a quick sanity regarding the appropriateness of the symmetric indefinite factorization and to provide useful bounds for the eigenvalues of $ \oobH_k $. From (\ref{eq:LBL}), $ \obH $ can be expressed as $ \obH = (\bP^\transpose\bL) \bB(\bP^\transpose\bL)^\transpose $. Then the second bullet point above can be shown easily by \textit{Sylvester's law of inertia}, which states that two congruent matrices have the same number of positive, negative and zero eigenvalues ($ \bA $ and $\bB $ are congruent if $ \bA = \bP\bB\bP^\transpose $ for some nonsingular matrix $ \bP $) \cite{sylvester1852xix}. From the third bullet point, if $ \obH $ is positive semidefinite, the corresponding $\bB$ is a diagonal matrix with nonnegative diagonal entries.

\subsection{Main Algorithms}
\label{subsec:algo}

We now illustrate how the $ \bL\bB\bL^\transpose $ factorization can be used in 2SPSA/2SG and discuss steps i) -- iii) in Sect. \ref{subsect:Introduction} in detail. The results are presented in three algorithms, with Algorithms~\ref{algo:two_rank_one_update}--\ref{algo:descent_direction} implementing steps \ref{item:ranktwomodification}\textendash \ref{item:descentdirection}, respectively. Algorithm~\ref{algo:2SP} in Subsection~\ref{subsec:complexity} is to produce the updated $\hbtheta_k$. Code for all algorithms is available at \url{https://github.com/jingyi-zhu/Fast2SPSA}.

\textbf{Two rank-one modifications} Although the direct calculation of $ \obH_k $ in (\ref{eq:two_rank_one_update}) only costs $ O(p^2) $, the subsequent preconditioning step incurs a computational cost of $ O(p^3) $ when not using any factorization of $ \obH_k $. Therefore, in anticipation of the subsequent necessary preconditioning, we propose to keep track of the $ \bL\bB\bL^\transpose $ factorization of $ \obH_k $ instead of the matrix itself. That is, the two direct rank-one modifications in (\ref{eq:two_rank_one_update}) are transformed to two non-trivial modifications on the $ \bL\bB\bL^\transpose $ factorization, which also incurs a computational cost of $ O(p^2) $. It is not necessary that $ \obH_k $ be explicitly computed in the algorithm, thereby avoiding the $ O(p^3) $ cost arising from matrix-associated necessary multiplications in the preconditioning. 

Lemma~\ref{lem:rank_one_update} states that the $ \bL\bB\bL^\transpose $ factorization can be updated for rank-one modification at a computational cost of $O(p^2)$. The detailed algorithm is established in \cite{sorensen1977updating}. We adopt that algorithm to our two rank-one modifications in (\ref{eq:two_rank_one_update}) and present the result in Theorem~\ref{thm:two_rank_one_update}.

\begin{lemma}\label{lem:rank_one_update}
	\textit{\cite[Thm. 2.1]{sorensen1977updating}}. 
	Let $ \bA \in \real^{p \times p} $ be symmetric (possibly indefinite) and \emph{non-singular} with $ \bP\bA\bP^\transpose = \bL\bB\bL^\transpose $. Suppose that $ \bz \in \real^p, \sigma \in \real $ are such that
	\begin{equation}\label{eq:IMF_rank_one_update}
	\tbA = \bA + \sigma\bz\bz^\transpose
	\end{equation}
	is also \emph{nonsingular}. Then the factorization $ \tilde{\bP}\tbA\tilde{\bP}^\transpose = \tilde{\bL}\tilde{\bB}\tilde{\bL}^\transpose $ can be obtained from the factorization $ \bP\bA\bP^\transpose = \bL\bB\bL^\transpose $ with a computational cost of $ O(p^2) $.
\end{lemma}

\begin{theorem}\label{thm:two_rank_one_update}
	Suppose $ \obH_k $ is given in (\ref{eq:two_rank_one_update}). Further assume that both $ \obH_{k-1} $ and $ \obH_k $ are \emph{nonsingular} and the factorization $ \bP_{k-1}\obH_{k-1}\bP_{k-1}^\transpose = \bL_{k-1}\bB_{k-1}\bL_{k-1}^\transpose $ is available. Then the factorization
	\begin{equation}\label{eq:IMF_rwo_rank_one_update}
	\bP_k\obH_k\bP_k^\transpose = \bL_k\bB_k\bL_k^\transpose
	\end{equation}
	can be obtained at a computational cost of $ O(p^2) $.
\end{theorem}

\begin{proof} 	 
	With Lemma~\ref{lem:rank_one_update}, we see that (\ref{eq:IMF_rwo_rank_one_update}) can be obtained by applying (\ref{eq:IMF_rank_one_update}) twice with $ \sigma = b_k, \bz = \tbu_k $ and $ \sigma = -b_k, \bz = \tbv_k $, respectively. Because each update requires a computational cost of $ O(p^2) $, the total computational cost remains $ O(p^2) $.
\end{proof}

\begin{remark}
	The nonsingularity (\emph{not} necessarily positive-definiteness) of $ \obH_k $ is a modest assumption for the following reasons: i) $ \obH_0 $ is often initialized to be a positive definite matrix satisfying the nonsingularity assumption, e.g., $ \obH_0 = c\bI $ for some constant $ c > 0 $. ii) Whenever $ \obH_k $ violates the nonsingularity assumption due to the two rank-one modifications in (\ref{eq:two_rank_one_update}), a new pair of $ \bDelta_k $ and $ \tbDelta_k $ along with the noisy measurements can be generated to redo the modifications in (\ref{eq:two_rank_one_update}). In practice, the singularity of $ \obH_k $ can be detected via the entry-wise bounds of $ \bL_k $ per \cite{bunch1977some}. Namely, if $ \bL_k $ has an entry exceeding $ 2.7808 $, the nonsingularity assumption of $ \obH_k $ is violated. It is indeed possible to compute the probability of getting a singular $ \obH_k $; however, we deem it as a minor practical issue and do not pursue further analysis in this work. iii) Because the second-order method is often recommended to be implemented only after $ \hbtheta_k $ reaches the vicinity of the optimal point $ \btheta^* $, and the true Hessian matrix of $ \btheta^* $ is assumed to be positive definite \cite{spall2000adaptive}, the estimate $ \obH_k $ is ``pushed" towards nonsingularity. The bottom line is that we are able to run second-order methods at any iteration $ k $, but are more interested when $\hbtheta_k$ is near $\btheta^*$.
\end{remark}

We summarize the two rank-one modifications of $ \obH_k $ in the following Algorithm~\ref{algo:two_rank_one_update}. The outputs of Algorithm~\ref{algo:two_rank_one_update} are used to obtain a computational cost of $ O(p^2) $ in the preconditioning step as the eigenvalue modifications on $ \bB_k $, a diagonal block matrix, is more efficient than the direct eigenvalue modifications in (\ref{eq:f_k_sqrtm}) and (\ref{eq:f_k_add}). Algorithm~\ref{algo:two_rank_one_update} is the key that renders steps ii) and iii) in Subsection~\ref{subsect:Introduction} achievable at a computational cost of $ O(p^2) $. 
\begin{algorithm}[!htbp]
	\caption{Two rank-one updates of $ \obH_k $}
	\label{algo:two_rank_one_update}
	\begin{algorithmic}[1]
		\renewcommand{\algorithmicrequire}{\textbf{Input:}}
		\renewcommand{\algorithmicensure}{\textbf{Output:}}
		\REQUIRE matrices $ \bP_{k-1}, \bL_{k-1}, \bB_{k-1} $ in the symmetric indefinite factorization of $\obH_{k-1}$, scalars $ t_k, b_k $, and vectors $ \bu_k, \bv_k $ computed per Table~\ref{table:u_k_v_k}.
		\ENSURE matrices $ \bP_k, \bL_k, \bB_k $ in the symmetric indefinite factorization of $ \obH_k $ per (\ref{eq:LBL}).
		\STATE \textbf{set} $ \bP_k \gets \bP_{k-1}, \bL_k \gets \bL_{k-1}, \bB_k \gets t_k\bB_{k-1} $. 
		\STATE \textbf{update} $ \bP_k, \bL_k, \bB_k $ with the rank-one modifications $ b_k\tbu_k\tbu_k^\transpose $ with $ \tbu_k $ computed in (\ref{eq:u_k_tilde}) and $ -b_k\tbv_k\tbv_k^\transpose $ with $ \tbv_k $ computed in (\ref{eq:v_k_tilde}), using the updating procedure outlined in \cite{sorensen1977updating}. (Recall that the code is available at \url{https://github.com/jingyi-zhu/Fast2SPSA}.)
		\RETURN matrices $ \bP_k, \bL_k, \bB_k $. 
	\end{algorithmic}
\end{algorithm}

\begin{remark}
	Though $ \obH_k $ is not explicitly computed during each iteration, whenever needed it can be computed easily from its $ \bL\bB\bL^\transpose $ factorization, though with a computational cost of $ O(p^3) $, i.e, $ \obH_k = \bP_k^\transpose\bL_k\bB_k\bL_k^\transpose\bP_k $. This calculation yields the same $ \obH_k $ as (\ref{eq:H_overline}) or (\ref{eq:two_rank_one_update}). The $ \bL\bB\bL^\transpose $ factorization of $ \obH_0 $ requires a computational cost of, at most, $ O(p^3) $ \cite[Table 2]{bunch1971direct}. However, as a one-time sunk-in cost, it does not compromise the overall computational cost. Of course, we can avoid this bothersome issue by initializing $ \obH_0 $ to a diagonal matrix, which immediately gives $ \bP_0 = \bL_0 = \bB_0 = \bI $. Generally, the cost for initialization is trivial if $\obH_0$ is a diagonal matrix.
\end{remark}

\textbf{Preconditioning}
Given the factorization of the estimated Hessian information $ \obH_k $, which is symmetric yet potentially \emph{indefinite} (especially during early iterations), we aim to output a factorization of the Hessian approximation $ \oobH_k $ such that $ \oobH_k $ is symmetric and sufficiently positive definite, i.e., $ \lambda_{\min}(\oobH_k) \geq \tau $ for some constant $ \tau > 0 $. With the above $ \bL\bB\bL^\transpose $ factorization associated with $\obH_k$ obtained from the previous two rank-one modification steps, we can modify the eigenvalues of $ \bB_k $. Note that $\bB_k$ is a block diagonal matrix, so any eigenvalue modification can carried out inexpensively. This is in contrast to directly modifying the eigenvalues of $\obH_k$ to obtain $ \oobH_k $, which is computationally-costly as laid out in Subsection~\ref{subsect:p3cost}. Denote $ \obB_k $ as the modified matrix from $ \bB_k $. Note that $ \oobH_k $ and $ \obB_k $ are congruent as $ \oobH_k = (\bP_k^\transpose\bL_k)\obB_k(\bP_k^\transpose\bL_k)^\transpose $. By Sylvester's law of inertia, the positive definiteness of $ \oobH_k $ is guaranteed as long as $ \obB_k $ is positive definite. 

To modify the eigenvalues of $ \bB_k $, we borrow the ideas from the modified Newton's method \cite[pp. 50]{nocedal2006numerical} to set $ \lambda_j(\obB_k) = \max \set{\tau_k, |\uplambda_j(\bB_k)|} $ for $ j = 1, ..., p $, where $ \tau_k $ is a user-specified stability threshold, which is possibly data-dependent. A possible choice of the uniformly bounded $ \set{\tau_k} $ sequence in the Section~\ref{sec:numerical} is to set $ \tau_k = \max\{10^{-4}, 10^{-4}p\max_{1\leq j \leq p} |\uplambda_j(\bB_k)| \} $. The intuition behind the eigenvalue modification in Algorithm~\ref{algo:preconditioning} is to make $ \obB_k $ well-conditioned while behaving similarly to $ \bB_k $. The pseudo code of the preconditioning step is listed in Algorithm~\ref{algo:preconditioning}.
\begin{algorithm}[htbp]
	\caption{Preconditioning}
	\label{algo:preconditioning}
	\begin{algorithmic}[1]
		\renewcommand{\algorithmicrequire}{\textbf{Input:}}
		\renewcommand{\algorithmicensure}{\textbf{Output:}}
		\REQUIRE user-specified stability-threshold $ \tau_k > 0 $ and matrix $ \bB_k $ in the symmetric indefinite factorization of $ \obH_k $.
		\ENSURE matrix $ \bQ_k $ in the eigen-decomposition of $ \bB_k $ and the modified matrix $ \obLambda_k $.
		\STATE \textbf{apply} eigen-decomposition of $ \bB_k = \bQ_k \bLambda_k \bQ_k^\transpose $, where $ \bLambda_k = \text{diag}(\lambda_{k1}, ..., \lambda_{kp}) $ and $ \lambda_{kj} \equiv \lambda_j(\bB_k) $ for $ j = 1, ..., p $.
		\STATE \textbf{update} $ \obLambda_k = \text{diag}(\bar{\lambda}_{k1}, ..., \bar{\lambda}_{kp}) $ with $ \bar{\lambda}_{kj} = \max \set{\tau_k, |\uplambda_{kj}|} $ for $j = 1, ..., p$.
		\RETURN eigen-decomposition of $ \obB_k = \bQ_k \obLambda_k \bQ_k^\transpose $.
	\end{algorithmic}
\end{algorithm}

\begin{remark}
	Although the eigen-decomposition, in general, incurs an $ O(p^3) $ cost, the block diagonal structure of $ \bB_k $ allows such an operation to be implemented relatively inexpensively. In the worst-case scenario, $ \bB_k $ consists of $ p/2 $ diagonal blocks of size $ 2 \times 2 $, where eigen-decompositions are applied on each block separately leading to a total computational cost of $ O(p) $. For the sake of efficiency, the matrix $ \oobH_k $ is not explicitly computed. Whenever needed, however, it can be computed by $ \oobH_k = \bP_k^\transpose\bL_k\bQ_k\obLambda_k\bQ_k^\transpose\bL_k^\transpose\bP_k $ at a cost of $O(p^3)$. 
\end{remark}

Algorithm~\ref{algo:preconditioning} makes our approach different from \cite{spall2000adaptive}. We only modify the eigenvalues of $ \bLambda_k $ (or equivalently of $ \bB_k $), which indirectly affects the eigenvalues of $ \oobH_k $ in a non-trivial way. However, if one constructs $ \obH_k $ and $ \oobH_k $ from their factorization (formally unnecessary as mentioned above), Algorithm~\ref{algo:preconditioning} can be viewed as a function that maps $ \obH_k $ to a positive-definite $ \oobH_k $. In this sense, Algorithm~\ref{algo:preconditioning} is just a special choice of $ f_k(\cdot) $ in (\ref{eq:H_ooverline}) even though such a $ f _k(\cdot) $ is non-trivial and difficult to find. 

\textbf{Descent direction} After the preconditioning step, the descent direction $ \bd_k: \oobH_k \bd_k = \bG_k(\hbtheta_k) $ can be computed readily via one forward substitution with respect to (w.r.t.) the lower-triangular matrix $\bL_k$ and one backward substitution w.r.t. the upper-triangular matrix $\bL_k^\transpose$, as the decomposition $\oobH_k=\bP_k^\transpose\bL_k\bQ_k\obLambda_k\bQ_k^\transpose\bL_k^\transpose\bP_k $ is available. The estimate $ \hbtheta_k $ can then be updated as in (\ref{eq:theta_update_s}). Note that $ \oobH_k $ is not directly computed in any iteration, and the forward and backward substitutions are implemented through the terms in the $ \bL\bB\bL^\transpose $ factorization. Algorithm~\ref{algo:descent_direction} below summarizes the details.
\begin{algorithm}[H]
	\caption{Descent Direction Step}
	\label{algo:descent_direction}
	\begin{algorithmic}[1]
		\renewcommand{\algorithmicrequire}{\textbf{Input:}}
		\renewcommand{\algorithmicensure}{\textbf{Output:}}
		\REQUIRE gradient estimate $ \bG_k(\hbtheta_k) $, and matrices $ \bP_k, \bL_k, \bQ_k, \obLambda_k $ in the $ \bL\bB\bL^\transpose $ factorization of $ \oobH_k $.
		\ENSURE descent direction $ \bd_k $.
		\STATE \textbf{Solve} $ \bz$ by forward substitution such that $\bL_k\bz = \bP_k\bG_k(\hbtheta_k) $.
		\STATE \textbf{Compute} $ \bw $ such that $ \bw = \bQ_k\obLambda_k^{-1}\bQ_k^\transpose \bz $.
		\STATE \textbf{Solve} $ \by $ by backward substitution such that $ \bL_k^\transpose\by = \bw $.
		\RETURN $ \bd_k = \bP_k^\transpose\by $.
	\end{algorithmic}
\end{algorithm}
Given the triangular structure of $\bL_k$ and that both $\bP_k$ and $\bQ_k$ are permutation matrices, the computational cost of Algorithm~\ref{algo:descent_direction} is dominated by $O(p^2)$.

\subsection{Overall Algorithm (Second-Order SP) and Computational Complexity }
\label{subsec:complexity}

With the aforementioned steps, we present the \emph{complete} algorithm for implementing second-order SP in Algorithm~\ref{algo:2SP} below, which applies to 2SPSA/2SG/E2SPSA/E2SG. Complete computational complexity analysis for 2SPSA is also stated, and the suggestions for the user-specified inputs are listed in \cite[Sect. 7.8.2]{spall2005introduction}. Results for 2SG/E2SPSA/E2SG can be obtained similarly. 

\begin{algorithm}[!htbp]
	\caption{Efficient Second-order SP (applies to 2SPSA, 2SG, E2SPSA, and E2SG)}
	\label{algo:2SP}
	\begin{algorithmic}[1]
		\renewcommand{\algorithmicrequire}{\textbf{Input:}}
		\renewcommand{\algorithmicensure}{\textbf{Output:}}
		\REQUIRE initialization $ \hbtheta_0 $ and $ \bP_0, \bQ_0, \bB_0 $ in the symmetric indefinite factorization of $ \obH_0 $; user-specified stability-threshold $ \tau_k > 0 $; coefficients $ a_k, c_k, w_k $ and, for 2SPSA/E2SPSA, $ \tilde{c}_k $.
		\ENSURE terminal estimate $ \hbtheta_k $.
		\STATE \textbf{set} iteration index $ k = 0 $.
		\WHILE{terminating condition for $ \hbtheta_k $ has not been satisfied}
		\STATE \textbf{estimate} gradient $ \bG_k(\hbtheta_k) $ by (\ref{eq:gradient_estimate_2SPSA}) or (\ref{eq:gradient_estimate_2SG}). 
		\STATE \textbf{compute} $ t_k, b_k, \tbu_k $ and $ \tbv_k $ by (\ref{eq:u_k_tilde}), (\ref{eq:v_k_tilde}) and Table~\ref{table:u_k_v_k}.
		\STATE \textbf{update} the symmetric indefinite factorization of $ \obH_k $ by Algorithm~\ref{algo:two_rank_one_update}.
		\STATE \textbf{update} the symmetric indefinite factorization of $ \oobH_k $ by Algorithm~\ref{algo:preconditioning}.
		\STATE \textbf{compute} the descent direction $ \bd_k $ by Algorithm~\ref{algo:descent_direction}.
		\STATE \textbf{update} $ \hbtheta_{k+1} = \hbtheta_k - a_k\bd_k $.
		\STATE $ k \leftarrow k + 1 $.
		\ENDWHILE
		\RETURN $ \hbtheta_k $.
	\end{algorithmic}
\end{algorithm}

For the terminating conditions, the algorithm is set to stop when a pre-specified total number of function (applicable for 2SPSA and E2SPSA) or gradient (applicable for 2SG and E2SG) measurements is reached or the norms of the differences between several consecutive estimates are less than a pre-specified threshold. Note that, for each iteration, four noisy loss function measurements are required in the gradient-free case and three noisy gradient measurements are required in the gradient-based case.

The corresponding computational complexity analysis for Algorithm~\ref{algo:2SP} under the gradient-free case is summarized in Table~\ref{table:computational_complexity}. Analogously, the analysis can be carried out for the gradient-based case (2SG) and the feedback-based case (E2SPSA or E2SG). A floating-point operation is assumed to be either a summation or a multiplication, while transposition requires no FLOPs. For the updating $ \obH_k $ step in original 2SPSA, $ 3p^2 $ FLOPs are required per (\ref{eq:H_overline}) and $ 4p^2 $ FLOPs are required per (\ref{eq:H_hat}). In the proposed implementation, $ 10p $ FLOPs are required to get $ \tbu_k $ and $ \tbv_k $ per (\ref{eq:u_k_tilde}) and (\ref{eq:v_k_tilde}), respectively, and $ 22p^2/6 + O(p) $ FLOPs are required to update the symmetric indefinite factorization of $ \obH_k $ \cite[Thm. 2.1 ]{sorensen1977updating}. For the preconditioning step in original 2SPSA, if using (\ref{eq:f_k_sqrtm}), $ p^3 + p $ FLOPs are required to get $ \obH_k\obH_k + \delta_k\bI $ and additional $ 50p^3/3 + O(p^2) $ FLOPs are required for the matrix square root operation \cite{higham1987computing}. In the proposed implementation, at most $ 7p $ FLOPs are required to get an eigenvalue decomposition on $ \bB_k $ ($ 14 $ FLOPs for at most $ p/2 $ blocks of size $ 2 \times 2 $) and $ p $ FLOPs are required to update the eigenvalues of $ \bB_k $. For computing the descent direction $ \bd_k $ in the original 2SPSA, $ p^3/3 $ FLOPs are required to apply Cholesky decomposition for $ \oobH_k $ and $ 2p^2 $ FLOPs are required for the backward substitutions. In the proposed implementation, $ 4p^2 + 2p $ FLOPs are required to backward substitutions.


\begin{table}[!htbp]
	\caption{Computational complexity analysis in gradient-free case (2SPSA in Algorithm~\ref{algo:2SP}) Complexity cost shown in FLOPs.}
	\centering
	\begin{tabular}{|c|c|c|}
		\hline
		Leading Cost & Original 2SPSA & Proposed Implementation\\
		\hline\hline
		Update $ \obH_k $ & $ 7p^2 $ & $ 3.67p^2 + O(p) $ \\
		\hline
		Precondition $ \oobH_k $ & $ 17.67p^3 + O(p^2) $ & $ 8p $ \\
		\hline
		Descent direction $ \bd_k $ & $ 0.33p^3 + O(p^2) $ & $ 4p^2 + O(p) $ \\
		\hline\hline
		Total Cost & $ 18p^3 + O(p^2) $ & $7.67p^2 + O(p) $ \\
		\hline
	\end{tabular}\label{table:computational_complexity}
\end{table}

Table~\ref{table:computational_complexity} may not provide the lowest possible computational complexities, because a great deal of existing work on parallel computing\textemdash such as \cite{george1986parallel} on parallelization of Cholesky decomposition, \cite{deadman2012blocked} for computing principal matrix square root, and \cite{dongarra1987fully} for the symmetric eigenvalue problem\textemdash have tremendously accelerated the matrix-operation computing speed in modern data analysis packages. Nonetheless, even with such enhancements, the FLOPS counts remain $ O(p^3) $ in the standard methods. The bottom line is that our proposed implementation reduces the overall computational cost from $ O(p^3) $ to $ O(p^2) $.

\section{Theoretical Results and Practical Benefits}\label{sec:theory}

This section presents the theoretical foundation related to the almost sure convergence and the asymptotic normality of $ \hbtheta_k $. We also offer comments on the practical benefits from the proposed scheme. Lemma~\ref{lem:Ostrowski} provides the theoretical guarantee to connect the eigenvalues of $ \oobH_k $ and $ \obLambda_k $, which are important for proving Theorem~\ref{thm:H_barbar_property}--\ref{thm:H_barbar_uniform_bound} related to the matrix properties of $\obH_k$ and $\oobH_k$. 

\begin{lemma}\label{lem:Ostrowski}
	\textit{\cite[Thm. 4.5.9]{horn1990matrix}} Let $ \bA, \bS \in \real^{p \times p} $, with $ \bA $ being symmetric and $ \bS $ being \emph{nonsingular}. Let the eigenvalues of $ \bA $ and $ \bS\bA\bS^\transpose $ be arranged in \emph{nondecreasing} order. Let $ \sigma_1 \geq \cdots \geq \sigma_p > 0 $ be the singular values of $ \bS $. For each $ j = 1, \cdots, p $, there exists a positive number $ \zeta_j \in [\sigma_p^2, \sigma_1^2] $ such that $ \lambda_j(\bS\bA\bS^\transpose) = \zeta_j\lambda_j(\bA)$.
\end{lemma}

Before presenting the main theorems, we first discuss the singular values of $ \bL_k $. Denote $ \{\sigma_i(\bL_k)\}_{i=1}^p $ as the singular values of $ \bL_k $. Also let $ \sigma_{\min}(\cdot) = \min_{1\leq i \leq p} \sigma_i(\cdot)$ and $ \sigma_{\max}(\cdot) = \max_{1\leq i \leq p} \sigma_i(\cdot) $. Since $ \bL_k $ is a unit lower triangular matrix, we have $ \lambda_j(\bL_k) = 1 $ for $ j = 1, .., p $ and $ \det(\bL_k) = 1 $. From the entry-wise bounds of $ \bL_k $ in Subsection~\ref{subsec:IMF}, we see that $ p \leq \norm{\bL_k}_F \leq 3p^2/2 - p/2 $ for all $ k $, where $\norm{\cdot}_F$ is the Frobenius norm of the argument matrix in $ \real^{p\times p} $. With the lower bound of $ \sigma_{\min}(\bL_k) $ \cite{yi1997note}, there exists a constant $ \underline{\sigma} > 0 $ such that $ \sigma_{\min}(\bL_k) \geq \underline{\sigma} $ for all $ k $. On the other hand, by the equivalence of the matrix norms, i.e, $ \sigma_{\max}(\bL_k) = \norm{\bL_k}_2 \leq \norm{\bL_k}_F $ for $ \norm{\cdot}_2 $ being the spectral norm, there exists a constant $ \overline{\sigma} > 0 $ such that $ \sigma_{\max}(\bL_k) \leq \overline{\sigma} $ for all $ k$. Both $ \underline{\sigma} $ and $ \overline{\sigma} $ are independent of the sample path for $ \bL_k $. By the Rayleigh-Ritz theorem \cite[Thm. 4.2.2]{horn1990matrix}, $ \be_1^\transpose(\bL_k\bL_k^\transpose)\be_1 = 1 $ implies that $ \sigma_{\min}(\bL_k) \leq 1 $ and $ \sigma_{\max}(\bL_k) \geq 1 $. Combined, all the singular values of $ \bL_k $ are bounded uniformly across $k$, i.e., $ \underline{\sigma} < \sigma_{\min}(\bL_k) \leq 1 \leq \sigma_{\max}(\bL_k) \leq \overline{\sigma} $. Let $ \kappa(\bL_k) $ be the condition number of $ \bL_k $, then $ 1 \leq \kappa(\bL_k) \leq \overline{\sigma} / \underline{\sigma} $.

Because the focus of Algorithm~\ref{algo:preconditioning} is to generate a positive definite $ \obB_k $ (or equivalently its eigen-decomposition), we replace $ \tau_k $ in Theorem~\ref{thm:H_barbar_property}--\ref{thm:H_barbar_uniform_bound} with some constant $ \underline{\tau}\in\left(0,\uptau_k\right] $ independent of the sample path for $ \bB_k $ for all $k$. Note that the substitution is solely for succinctness and does not affect the theoretical result that $\obB_k$ is positive definite. Theorem~\ref{thm:H_barbar_property} presents the key theoretical properties of $ \oobH_k $ satisfying the regularity conditions in \cite[C.6]{spall2000adaptive}. Based on Theorem~\ref{thm:H_barbar_property}, the strong convergence, $ \hbtheta_k \to \btheta^* $ and $ \obH_k \to \bH(\btheta^*) $, can be established conveniently, see Remark~\ref{rmk:strong_convergence}. 

\begin{theorem}\label{thm:H_barbar_property} Assume there exists a symmetric indefinite factorization $ \obH_k = \bP_k^\transpose\bL_k\bB_k\bL_k^\transpose\bP_k$.
	Given any constant $\underline{\uptau} \in\left(0,\uptau_k\right] $ for all $ k $, the matrix $ \oobH_k = \bP_k^\transpose\bL_k\bQ_k\obLambda_k\bQ_k^\transpose\bL_k^\transpose\bP_k $ with $ \bQ_k$ and $\obLambda_k $ returned from Algorithm~\ref{algo:preconditioning} satisfies the following properties:
	\begin{itemize}
		\item[(a)] $ \lambda_{\min}(\oobH_k) \geq \underline{\sigma}^2\underline{\tau} > 0 $.
		\item[(b)] $\oobH_k^{-1}$ exists a.s., $ c_k^2 \oobH_k ^{-1}\to \bm{0}$ a.s., and for some constants $\updelta, \uprho>0$, $ \mathbb{E} [ \| \oobH_k^{-1} \|^{2+\updelta} ]\le \uprho $.
	\end{itemize}
\end{theorem}

\begin{proof}
	For all $k$, it is easy to see that $ \lambda_{\min}(\obLambda_k) \geq \underline{\tau} > 0 $ implying $ \obLambda_k $ is positive definite. Since both $ \bQ_k $ and $ \bL_k $ are nonsingular, by Sylvester's law of inertia \cite{sylvester1852xix}, $ \oobH_k $ is also positive definite as $ \obLambda_k $ is positive definite. Moreover, by Lemma~\ref{lem:Ostrowski},
	\begin{equation}\label{eq:lambda_min_H_barbar}
	\lambda_{\min}(\oobH_k) \geq \sigma_{\min}^2(\bL_k)\lambda_{\min}(\obLambda_k) \geq \underline{\sigma}^2\underline{\tau} > 0 .
	\end{equation}
	Since $ \oobH_k $ has a constant lower bound for all its eigenvalues across $k$, property (b) follows.
\end{proof}

\begin{remark}\label{rmk:strong_convergence}
	Theorem~\ref{thm:H_barbar_property} guarantees that $ \oobH_k $ is positive definite, and therefore the estimates of $\btheta$ in the second-order method move in a descent-direction on average. Meeting property (b) is also necessary in showing the convergence results. Suppose the standard regularity conditions in \cite[Sect. III and IV]{spall2000adaptive} hold. To show the strong convergence, $ \hbtheta_k \to \btheta^* $ and $ \obH_k \to \bH(\btheta^*) $, we only need to verify that $ \oobH_k $ satisfies the regularity conditions in \cite[C.6]{spall2000adaptive} because the key difference between the original 2SPSA/2SG and our proposed method is effectively the preconditioning step. Theorem~\ref{thm:H_barbar_property} verifies the Assumption C.6 in \cite{spall2000adaptive} directly, and therefore we have $ \hbtheta_k \to \btheta^* $ a.s. and $ \obH_k \to \bH(\btheta^*) $ a.s. under both the 2SPSA and 2SG settings by \cite[Thms. 1 and 2]{spall2000adaptive}. 
\end{remark}

Theorem~\ref{thm:H_bar_property} discusses the connection between $ \obH_k $ and $ \oobH_k $ when $ k $ is sufficiently large. It also verifies a key condition when proving the asymptotic normality of $ \hbtheta_k $, see Remark~\ref{rmk:asymptotic_normality}.

\begin{theorem}\label{thm:H_bar_property}
	Assume $ \bH(\btheta^*) $ is positive definite. When choosing $ 0 < \underline{\tau} \leq \lambda_{\min}(\bH(\btheta^*)) / (2\overline{\sigma}^2) $, there exists a constant $ K_1 $ such that for all $ k > K_1 $, we have $ \oobH_k = \obH_k $.
\end{theorem}

\begin{proof}
	By Remark~\ref{rmk:strong_convergence}, since $ \obH_k \to \bH(\btheta^*) $ a.s., there exists an integer $K_1$ such that for all $ k > K_1 $, $ \lambda_{\min}(\obH_k) \geq \lambda_{\min}(\bH(\btheta^*)) / 2 > 0 $. By Lemma~\ref{lem:Ostrowski}, we can achieve a lower bound for the eigenvalues of $ \bLambda_k $ as
	\begin{equation*}
	\lambda_{\min}(\bLambda_k) \geq \frac{\lambda_{\min}(\obH_k)}{\sigma_{\max}^2(\bL_k)} \geq \frac{\lambda_{\min}(\obH_k)}{\overline{\sigma}^2} \geq \underline{\tau}.
	\end{equation*}
	Therefore, for all $ k > K_1$, $ \obLambda_k = \bLambda_k $ and consequently $ \oobH_k = \obH_k $.
\end{proof}

\begin{remark}\label{rmk:asymptotic_normality}
	Theorem~\ref{thm:H_bar_property} shows that when $ k $ is large (the estimated Hessian $ \obH_k $ is sufficiently positive definite), the proposed preconditioning step will automatically make $ \oobH_k = \obH_k $, which satisfies one of the key required conditions for the asymptotic normality of $ \hat{\btheta}_k $ in \cite{spall2000adaptive}. Besides the additional regularity conditions in \cite[C.10--12]{spall2000adaptive}, we are required to verify that
	$
	\oobH_k - \obH_k \to \bm{0}~\text{a.s.},
	$
	which can be inferred by Theorem~\ref{thm:H_bar_property}. Following \cite[Thm. 3]{spall2000adaptive}, when the gain sequences have the standard form $ a_k = a/(A+k+1)^\alpha $ and $ c_k = c/(k+1)^\gamma $, the asymptotic normality of $ \hbtheta_k $ gives
	\begin{equation*}
	\begin{split}
	k^{(\alpha-2\gamma)/2}(\hbtheta_k - \btheta^*) \stackrel{\text{dist}}{\longrightarrow} N(\bm{\upmu}, \bm{\Omega}) &\quad \text{for 2SPSA,}\\
	k^{\alpha/2}(\hbtheta_k - \btheta^*) \stackrel{\text{dist}}{\longrightarrow} N(\bm{0}, \bm{\Omega'}) &\quad\text{for 2SG,}
	\end{split}
	\end{equation*}
	where the specifications of $ \alpha, \gamma, \bm{\upmu}, \bm{\Omega}$ and $ \bm{\Omega'} $ are available in \cite{spall2000adaptive}. Under the E2SPSA/E2SG settings, the convergence and asymptotic results can be derived analogously from \cite[Thms. 1--4]{spall2009feedback}.
\end{remark}

Because an ill-conditioned matrix may cause an excessive step-size in recursion (\ref{eq:theta_update_s}), leading to slow convergence \cite{li2018preconditioned}, we need to make sure that the resulting $\oobH_k $ (or its equivalent factorization) is not only positive definite but also numerically favorable. Theorem~\ref{thm:H_barbar_uniform_bound} below shows that changing the eigenvalues of $ \bLambda_k $ does not lead to the eigenvalues of $ \oobH_k $ becoming either too large or too small.

\begin{theorem}\label{thm:H_barbar_uniform_bound} Assume the eigenvalues of $\bH\parenthese{\btheta^*}$ are bounded uniformly such that $0< \underline{\uplambda}^*<\abs{\uplambda_j\parenthese{\bH\parenthese{\btheta^*}}}<\overline{\uplambda}^*<\infty $ for $j=1,...,p$ for all $k$. Then there exists some $K_2$ such that for $ k>K_2 $, the eigenvalues and condition number of $ \oobH_k $ are also bounded uniformly.
\end{theorem}

\begin{proof}
	Again by Remark~\ref{rmk:strong_convergence}, since $\obH_k\to \bH\parenthese{\btheta^*}$ a.s., therefore for all $k>K_2$, the eigenvalues of $ \obH_k $ are bounded uniformly in the sense that $ \underline{\lambda} < |\lambda_j(\obH_k)| < \overline{\lambda} $ for $ j = 1, ..., p $, where $ \underline{\lambda} = \underline{\uplambda}^*/2 $ and $ \overline{\lambda} =2 \overline{\uplambda}^*$ are constants independent of the sample path for $ \obH_k $. 
	Given $ \obH_k = \bP_k\bL_k\bB_k\bL_k^T\bP_k $, by Lemma~\ref{lem:Ostrowski},
	\begin{equation*}
	\frac{\lambda_{\min}(\obH_k)}{\sigma_{\max}^2(\bL_k)} \leq \lambda_{\min}(\bB_k) \leq \frac{\lambda_{\min}(\obH_k)}{\sigma_{\min}^2(\bL_k)},
	\end{equation*}
	and
	\begin{equation*}
	\frac{\lambda_{\max}(\obH_k)}{\sigma_{\max}^2(\bL_k)} \leq \lambda_{\max}(\bB_k) \leq \frac{\lambda_{\max}(\obH_k)}{\sigma_{\min}^2(\bL_k)}.
	\end{equation*}
	Similarly, since $ \oobH_k = \bP_k\bL_k\obB_k\bL_k^T\bP_k $,
	\begin{equation*}
	\begin{split}
	\lambda_{\min}(\oobH_k) &
	\geq \sigma_{\min}^2(\bL_k) \lambda_{\min}(\obB_k)\\
	&\geq \sigma_{\min}^2(\bL_k)\max\left\{\underline{\tau},\frac{\lambda_{\min}(\obH_k)}{\sigma_{\max}^2(\bL_k)}\right\}\\
	&
	\geq \underline{\sigma}^2\max\left\{\underline{\tau},\frac{\underline{\lambda}}{\overline{\sigma}^2}\right\},
	\end{split}
	\end{equation*}
	\begin{equation*}
	\begin{split}
	\lambda_{\max}(\oobH_k)&
	\leq \sigma_{\max}^2(\bL_k) \lambda_{\max}(\obB_k)\\
	&\leq \sigma_{\max}^2(\bL_k)\max\left\{\underline{\tau},\frac{\lambda_{\max}(\obH_k)}{\sigma_{\min}^2(\bL_k)}\right\}\\
	&\leq \overline{\sigma}^2\max\left\{\underline{\tau},\frac{\overline{\lambda}}{\underline{\sigma}^2}\right\},
	\end{split}
	\end{equation*}
	where $ \kappa(\cdot) $ is the condition number of the matrix argument. Since $ \underline{\sigma}^2, \overline{\sigma}^2, \underline{\lambda} $, and $ \overline{\lambda} $ are all constants specified before running the algorithm, the eigenvalues of $ \oobH_k $ are bounded uniformly across $ k > K_2 $.
	
	Moreover, for the condition number of $ \oobH_k $, we have
	\begin{equation*}
	\kappa(\oobH_k)
	\leq \frac{\sigma_{\max}^2(\bL_k)}{\sigma_{\min}^2(\bL_k)}\frac{\max\left\{\underline{\tau},\lambda_{\max}(\obH_k)/\sigma_{\min}^2(\bL_k)\right\}}{\max\left\{\underline{\tau},\lambda_{\min}(\obH_k)/\sigma_{\max}^2(\bL_k)\right\}}.
	\end{equation*}
	Hence the condition number of $ \oobH_k $ is also bounded uniformly across $k>K_2$.
\end{proof}

\begin{remark}
	Theorem~\ref{thm:H_barbar_uniform_bound} is highly desired for the preconditioning step since it is to ensure the numerical stability. Recall that the preconditioning step listed in Algorithm~\ref{algo:preconditioning} modifies the eigenvalues of $ \obH_k $ by modifying the eigenvalues of $ \bB_k $. This modification is desirable since the eigenvalues of $ \oobH_k $ are controllable, i.e., a bound for $ \lambda_j(\oobH_k) $ uniformly for sufficiently large $k$ under a given size $ p $ can be obtained. The controlled condition number in Theorem~\ref{thm:H_barbar_uniform_bound} differs from the original preconditioning procedure as in Eq. (\ref{eq:f_k_add}), which does not control the condition number of $\overline{\overline{\bH}}_k$. 
	
\end{remark}

\section{Discussion} \label{sec:discussion}
This short section discusses two practical questions regarding Algorithm~\ref{algo:2SP} that produces the updated estimate of $\btheta$. 1) What is the difference between the standard adaptive SPSA-based method and the proposed algorithm if $ \obB_k $ (or $ \obH_k $) is sufficiently positive definite? 2) How to recover $ \obH_k $ at any $ k $?

In the ideal case, if $ \bB_k $ (or $ \obH_k $) is assumed to always be positive definite, the preconditioning step becomes unnecessary and we can directly set the symmetric indefinite factorization of $ \obH_k $ as the symmetric indefinite factorization of $ \oobH_k $, i.e., $ \bLambda_k = \obLambda_k $. In this scenario, the proposed method is identical to the original 2SPSA. However, because of the symmetric indefinite factorization, the overall computational cost remains at $ O(p^2) $ as in Table~\ref{table:computational_complexity} and it is still favorable relative to the original 2SPSA, which incurs a computational cost of $ O(p^3) $ due to the Gaussian elimination of $ \oobH_k $ in computing the descent direction $ \bd_k $. As mentioned in Sect. \ref{subsect:p3cost}, however, \cite{rastogi2016efficient} uses the matrix inversion lemma to show that the computational cost can be reduced to $ O(p^2) $ as well. Comparing with \cite{rastogi2016efficient}, which directly updates the matrix $ \obH_k^{-1} $ using the matrix inverse lemma, our proposed method has more control over the eigenvalues of $ \obH_k $ and performs well even when $ \obH_k $ is ill-conditioned.

The second aspect is that both $ \obH_k $ and $ \oobH_k $ are never explicitly computed during each iteration. By maintaining the corresponding factorization, we avoid the expensive matrix multiplications and gain a much faster way to achieve second-order convergence. However, whenever needed, either $ \obH_k $ or $ \oobH_k $ can be directly computed from the factorizations at a cost of $ O(p^3) $, see Subsection~\ref{subsec:algo}. 

\section{Numerical studies} \label{sec:numerical}

In this section, we demonstrate the strength of the proposed algorithms by minimizing the skewed-quartic function \cite{spall2000adaptive} using the efficient 2SPSA/E2SPSA and training a neural network using the efficient 2SG.

\subsection{Skewed-Quartic Function}

We consider the following skewed-quartic function used in \cite{spall2000adaptive} to show the performance of the efficient 2SPSA/E2SPSA:
\begin{equation*}
L(\btheta) = \btheta^\transpose\bB^\transpose\bB\btheta+0.1
\sum_{i=1}^{p} (\bB\btheta)_i^3 +0.01 \sum_{i=1}^{p}
(\bB\btheta)_i^4,
\end{equation*}
where $ (\cdot)_i $ is the $i$-th component of the argument vector, and $ \bB $ is such that $ p\bB $ is an upper triangular matrix of all $1$'s. The additive noise in $ y(\cdot) $ is independent $\mathcal{N}\parenthese{0,0.05^2}$, i.e., $ y(\btheta) = L(\btheta) + \varepsilon$, where $ \varepsilon \sim \mathcal{N}\parenthese{0,0.05^2} $. It is easy to check that $ L(\btheta) $ is strictly convex with a unique minimizer $ \btheta^* = \bzero $ such that $L(\btheta^*) = 0$.

For the preconditioning step in the original 2SPSA/E2SPSA, we choose $ \oobH_k = \bm{f}_k(\obH_k) = (\obH_k\obH_k + 10^{-4}e^{-k} \bI)^{1/2} $, which satisfies the definition of $ \bm{f}_k(\cdot) $ in (\ref{eq:f_k_sqrtm}) since $ \delta_k = 10^{-4}e^{-k} \to 0 $. In the efficient 2SPSA/E2SPSA, we choose $ \obLambda_k = \text{diag}(\bar{\lambda}_{k1}, ..., \bar{\lambda}_{kp}) $ with $ \bar{\lambda}_{kj} = \max \{10^{-4}$, $10^{-4}p\max_{1\leq i \leq p} |\lambda_{ki}|, |\lambda_{kj}|\} $ for all $ j $, which is consistent with the suggestion in \cite[pp. 118]{sorensen1977updating} and satisfies Theorem~\ref{thm:H_barbar_property}. To guard against unstable steps during the iteration process, a blocking step is added to reset $ \hbtheta_{k+1} $ to $ \hbtheta_k $ if $ \norm{\hbtheta_{k+1} - \hbtheta_k} \geq 1 $. We choose an initial value $ \hbtheta_0 = [1, 1, \dots, 1]^\transpose $. 

We show three plots below. Figures \ref{fig:loss_2SPSA} and \ref{fig:loss_E2SPSA} illustrate how the efficient method here provides essentially the same solution in terms of the loss function values as the $O(p^3)$ methods in \cite{spall2000adaptive} and \cite{spall2009feedback} (2SPSA and feedback and weighting-based E2SPSA). Figure 4 illustrates how the $O(p^3)$ vs. $O(p^2)$ FLOPS-based cost in Table~\ref{table:computational_complexity} above is manifested in overall runtimes.

Figure~\ref{fig:loss_2SPSA} plots the normalized loss function values $ [L(\hbtheta_k) - L(\btheta^*)] / [L(\hbtheta_0) - L(\btheta^*)] $ of the original 2SPSA and the efficient 2SPSA averaged over 20 independent replicates for $ p = 100 $ and number of iterations $ N = 50,000 $. Similar to the numerical studies in \cite{spall2009feedback}, the gain sequences of the two algorithms are chosen to be $ a_k = a/(A+k+1)^{0.602} $, $ c_k = \tilde{c}_k = c/(k+1)^{0.101} $, and $ w_k = w/(k+1)^{0.501} $ where $ a = 0.04, A = 1000, c = 0.05 $, and $ w = 0.01 $ following the standard guidelines in \cite{spall1998implementation}. 

\begin{figure}[!t]
	\centering
	\includegraphics[width=\linewidth]{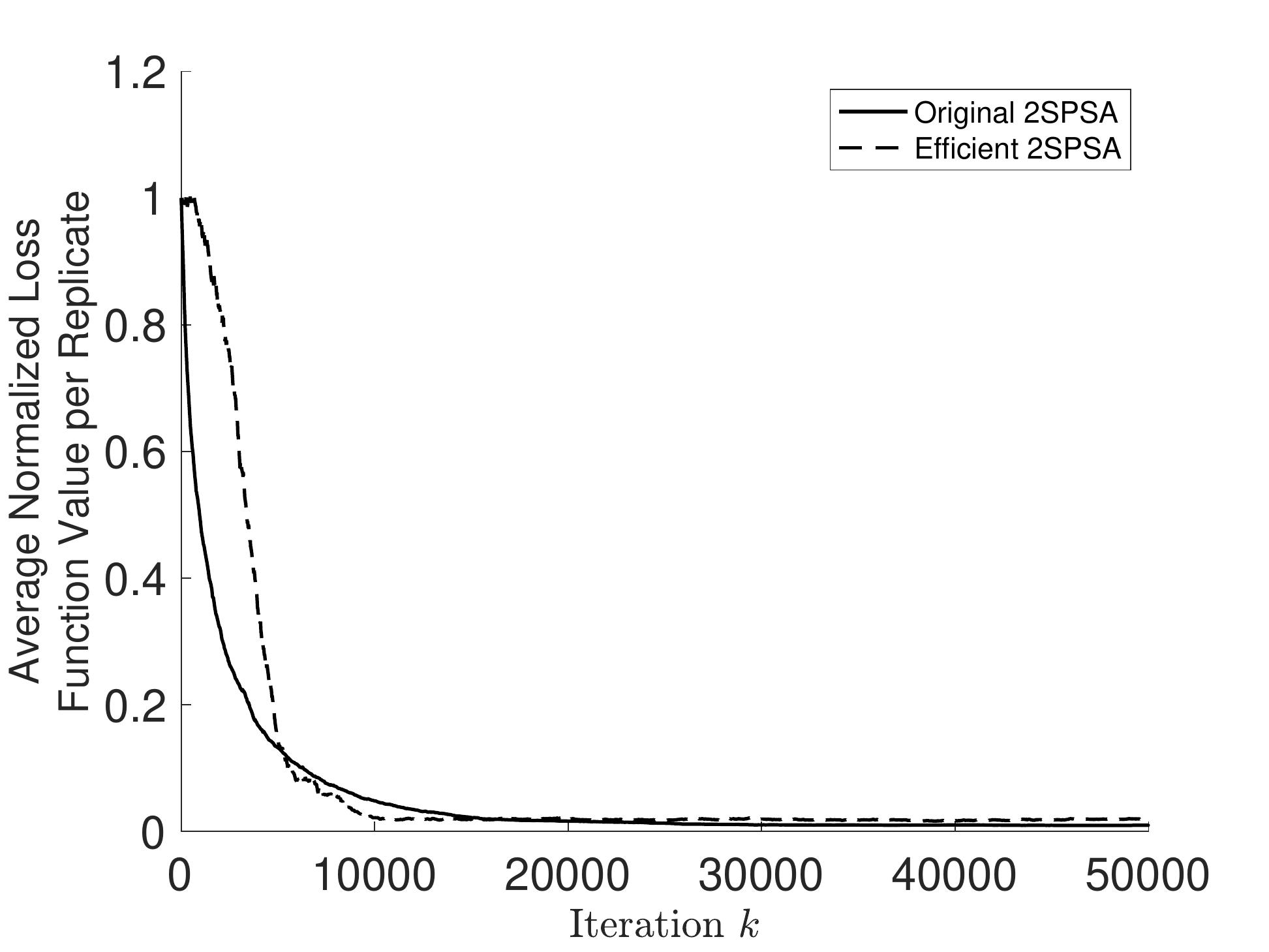}
	\caption{Similar performance of algorithms with respect to loss values (\emph{different} run times). Normalized terminal loss $ [L(\hbtheta_k) - L(\btheta^*)] / [L(\hbtheta_0) - L(\btheta^*)] $ of the original 2SPSA and the efficient 2SPSA averaged over 20 replicates for $ p = 100 $.}
	\label{fig:loss_2SPSA}
\end{figure}

Figure~\ref{fig:loss_E2SPSA} compares the normalized loss function values $ [L(\hbtheta_k) - L(\btheta^*)] / [L(\hbtheta_0) - L(\btheta^*)] $ of the standard E2SPSA and the efficient E2SPSA averaged over 10 independent replicates for $ p = 10 $ and number of iterations $ N = 10,000 $. The gain sequences of the two algorithms are chosen to have the form $ a_k = a/(A+k+1)^{0.602} $, $ c_k = \tilde{c}_k = c/(k+1)^{0.101} $, and $ w_k = w/(k+1)^{0.501} $ where $ a = 0.3, A = 50 $, and $ c = 0.05 $. The weight sequence $ w_k = \tilde{c}_k^2c_k^2/[\sum_{i=0}^{k}(\tilde{c}_i^2c_i^2 )]$ is set according to the optimal weight in \cite[Eq. (4.2)]{spall2009feedback}. 

\begin{figure}[!t]
	\centering
	\includegraphics[width=\linewidth]{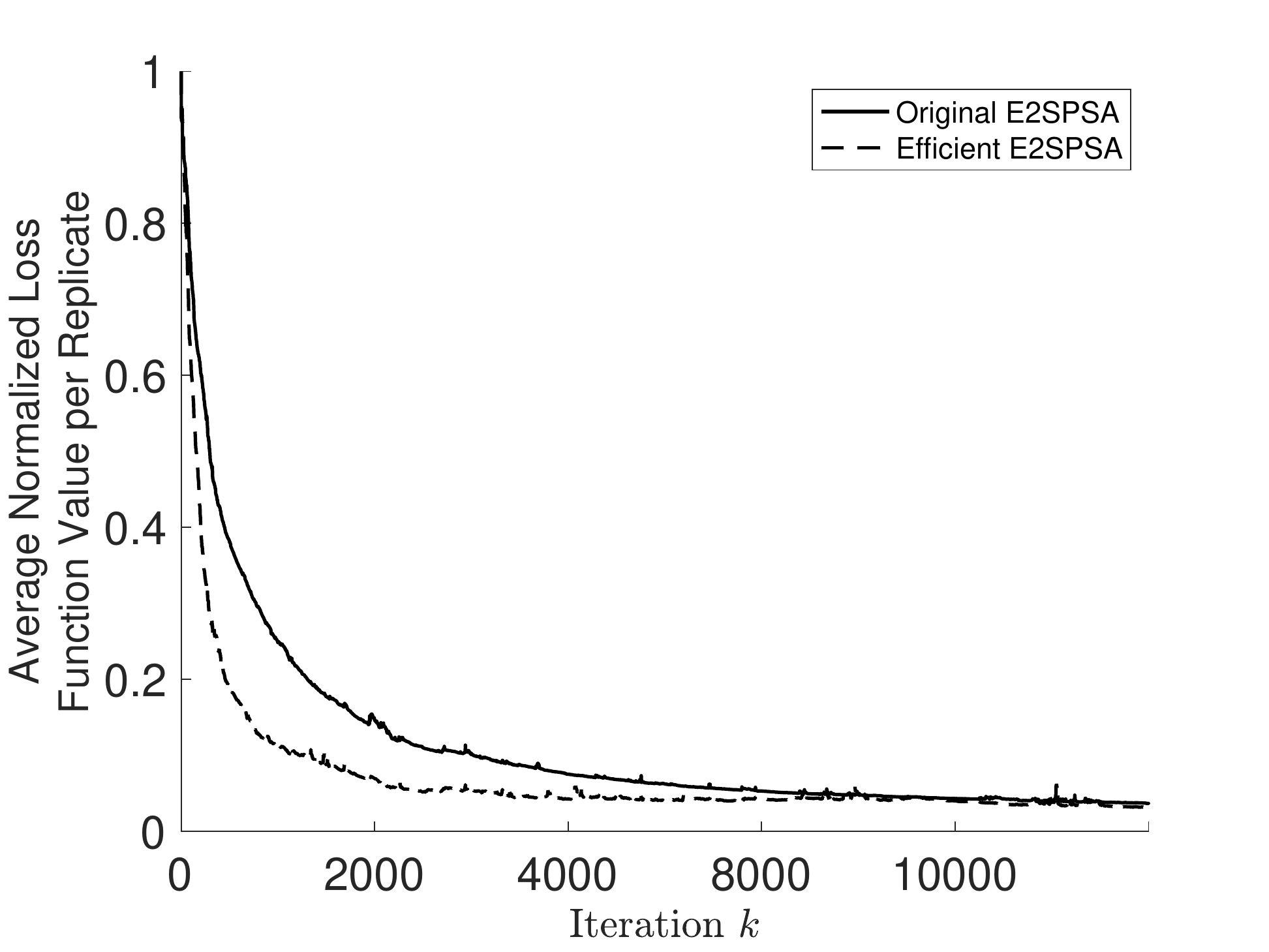}
	\caption{Similar performance of algorithms with respect to loss values (\emph{different} run times). Normalized terminal loss $ [L(\hbtheta_k) - L(\btheta^*)] / [L(\hbtheta_0) - L(\btheta^*)] $ of the original E2SPSA and the efficient E2SPSA averaged over 10 replicates for $ p = 10 $. }
	\label{fig:loss_E2SPSA}
\end{figure}

In the above comparisons, the loss function decreases significantly for all the dimensions with only noisy loss function measurements available. We see that the two implementations of E2SPSA provide close to the same accuracy for $1000$ or more iterations, although at a computing cost difference of $O(p^2)$ versus $O(p^3)$. Note that the differences (across $k$) between the original 2SPSA and the efficient 2SPSA/E2SPSA in Figure~\ref{fig:loss_E2SPSA} can be made arbitrarily small by picking an appropriate $ \bm{f}_k(\cdot) $ (or equivalently $ \oobH_k $) in the original 2SPSA, although such a choice might be non-trivial.

\begin{figure}[!t]
	\centering
	\includegraphics[width=\linewidth]{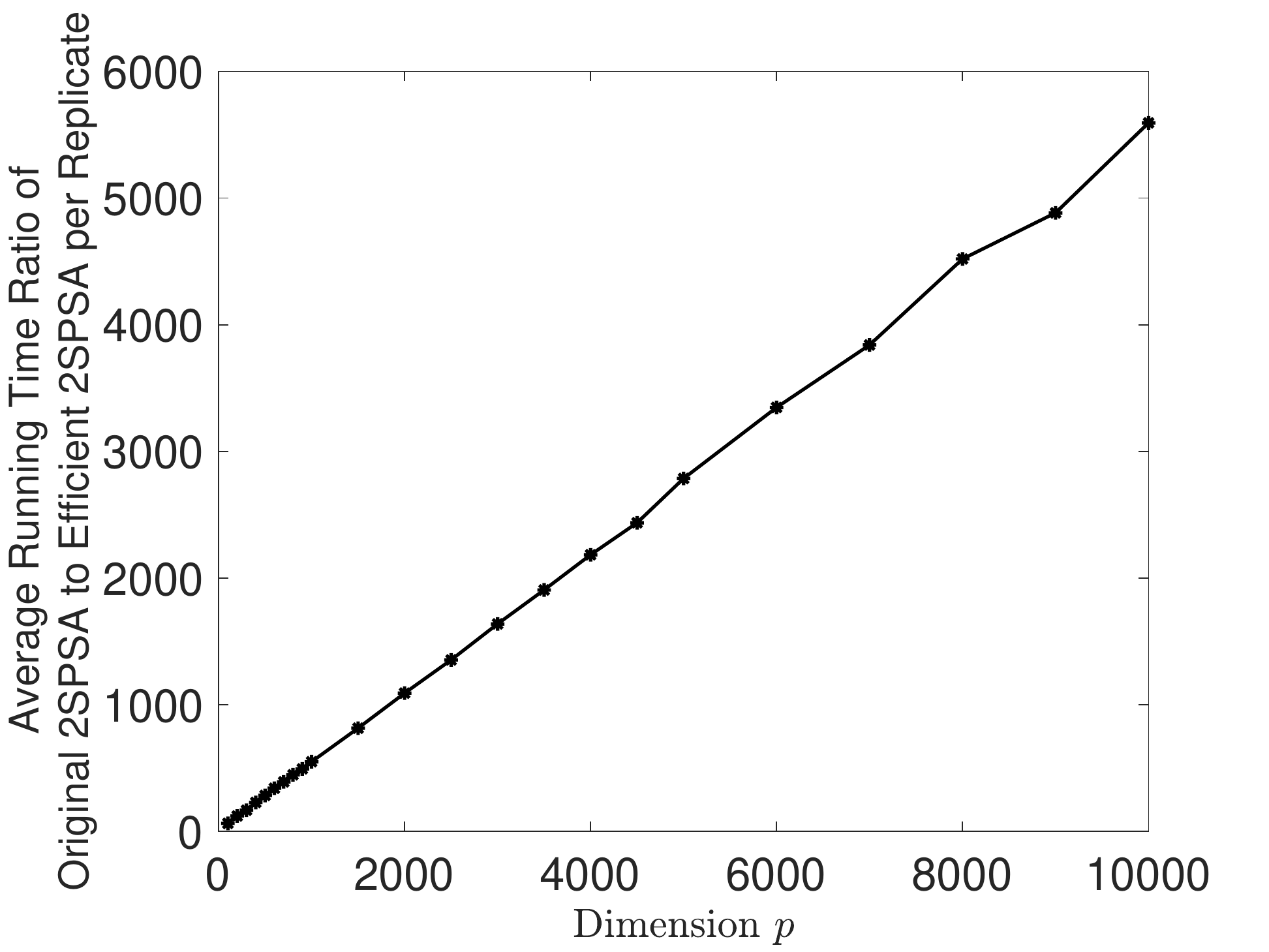}
	\caption{Running time ratio of the original 2SPSA to the efficient 2SPSA averaged over 10 replicates, where the same skewed-quartic loss function is used and the total number of iterations is fixed at 10 for each run. The trend is close to the theoretical linear relationship as a function of dimension $p$.}
	\label{fig:time}
\end{figure}
To measure the computational time, Figure~\ref{fig:time} plots the running time (measured by the built-in \textsc{C++} function \texttt{clock()} with no input) ratio of the original 2SPSA to the efficient 2SPSA averaged over 10 independent replicates with dimension up to 10000. It visualizes the practicality of the efficient 2SPSA over the original 2SPSA. In terms of the general trend, the linear relationship between the running time ratio and the dimension number is consistent with the $ O(p^3) $ cost for the original 2SPSA and $ O(p^2) $ cost for the efficient 2SPSA. From Figure~\ref{fig:time}, it is clear that the computational benefit of the efficient 2SPSA is more apparent as the dimension $ p $ goes up. The slope in Figure~\ref{fig:time} is roughly 0.56, which is consistent with the theoretical FLOPs ratio of 2.35 in Table~\ref{table:computational_complexity}, when accounting for differences due to the storage costs and code efficiency. With a more dedicated programming language, it is expected that the running time ratio will be closer to the theoretical FLOPs ratio in Table~\ref{table:computational_complexity}.

\subsection{Real-Data Study: Airfoil Self-Noise Data Set}

In this subsection, we compare the efficient 2SG with the stochastic gradient descent (SGD) and ADAM \cite{kingma2014adam} in training a one-hidden-layer feed-forward neural network to predict sound levels over an airfoil. Although there are many gradient-based methods to train a neural network, we select SGD and ADAM because they are popular and representative of algorithms within the machine learning community. Comparison of efficient 2SG and the two aforementioned algorithms is appropriate, as all of them use the noisy gradient evaluations \emph{only}, despite their different forms. Aside from the application here, neural networks have been widely used as function approximators in the field of aerodynamics and aeroacoustics. Recent applications include airfoil design \cite{rai2000aerodynamic} and aerodynamic prediction \cite{perez2000prediction}.

The dataset used in this example is the NASA data of NACA 0012 airfoil self-noise data set \cite{brooks1981trailing, brooks1989airfoil}, which is also available on the UC Irvine Machine Learning Repository \cite{brooks2014UCI}. This NASA dataset is obtained from a series of aerodynamic and acoustic tests of two and three-dimensional airfoil blade sections conducted in an anechoic wind tunnel. The inputs contain five variables: frequency (in Hertz); angle of attack (in degrees, not in radians); chord length (in meters); free-stream velocity (in meters per second); and suction side displacement thickness (in meters). The output contains the scaled sound pressure level (in decibels). Readers may refer to \cite{brooks1989airfoil} and \cite[Sect. 3]{errasquin2009airfoil} for further details.

Given the number of samples $ n = 1503 $, we fit the dataset using a one-hidden-layer neural network with 150 hidden neurons and sigmoid activating functions. Other choices of the neural network structures, such as using a different number of layers or different activation functions, have been implemented in \cite{errasquin2009airfoil}. Here, we use a neural network with a greater number of neurons than the one used in \cite{errasquin2009airfoil} to demonstrate the strength of the efficient 2SG in high-dimensional problems. The dimension $ p = 1051 $ is calculated as $ 5 \times 150 $ weights and $ 150 $ bias parameters for the hidden neurons plus 150 weights and 1 bias parameters for the output neuron.

Following the principles in \cite{wilson2003general}, we train the neural network in an online manner, where only one training sample is evaluated during each iteration. Denote the dataset as $ \{(y_i, \bx_i)\}_{i=1}^n $ and the parameters in the neural network as $ \btheta $. The loss function is chosen to be the empirical risk function (ERF), i.e., $ L(\btheta) = (1/n) \sum_{i=1}^n (y_i - \hat{y}_i)^2 $, where $ \hat{y}_i $ is the neural network output based on input $ \bx_i $ and parameter $ \btheta $. Consistent with the online training of an ERF in machine learning, the loss function based on that one training sample can be deemed as a noisy measurement of the loss function based on the entire dataset.

We implement SGD and ADAM with 10 epochs, each corresponds to 1503 iterations (one iteration per data point), resulting in a total of 15030 iterations. The gain sequence is chosen to be $ a_k = a / (k + 1 + A)^\alpha $ with $ A = 1503$ being 10\% of the total number of iterations and $ \alpha = 1 $ following \cite[pp. 113\textendash 114]{spall2005introduction}. After tuning for optimal performance, we choose $ a = 1 $ for SGD and ADAM \cite{kingma2014adam} . Other hyper-parameters for ADAM are determined from the default settings in \cite{kingma2014adam}. There is no ``re-setting'' of $ a_k $ imposed at the beginning of each epoch so that the gain sequence goes down consecutively across iterations and epochs. The initial value $ \hbtheta_0 = \bm{0} $. Recall that efficient 2SG requires three back-propagations per iteration, where SGD and ADAM only requires one back-propagation per iteration. Therefore, for a fair comparison, we implement the efficient 2SG under two different scenarios: (1) serial computing, and (2) concurrent computing.

Within each iteration of efficient 2SG, the three gradient measurements, $ \bY_k(\hbtheta_k),\bY_k(\hbtheta_k +c_k\bDelta_k) $ and $ \bY_k(\hbtheta_k-c_k\bDelta_k) $ can be computed simultaneously since they do not rely on each other. Using this concurrent implementation, the time spent in back-propagation can be reduced to one third of the original time. All the remaining steps are unchanged. Although the efficient 2SG takes time in performing algorithm~\ref{algo:preconditioning}, numerical studies indicate that majority of the time is spent on the back-propagation. Hence, under the concurrent implementation, the efficient 2SG has roughly the same running time per iteration as SGD and ADAM. Figure~\ref{fig:real_data_log_MSE_per_iteration} shows the value of ERF under the concurrent implementation. In the efficient 2SG, the gain sequences are chosen to be $ a_k = a/(A+k+1)^\alpha $, $ w_k = 1/(k+1) $ and $ c_k = c/(k+1)^\gamma $ with $A=1503, \alpha=1 $ and $ \gamma = 1/6 $ following \cite{spall1998implementation}. Other parameters $ a = 0.1 $ and $ c = 0.05 $ are tuned for optimal performance. The matrix $ \obLambda_k $ is computed the same as in the skewed-quartic function above. For better practical performance, training data is normalized to the range $ [0,1] $. Since all the inputs and output are positive, the normalization is simply done by dividing the data by their corresponding maximum. Figure~\ref{fig:real_data_log_MSE_per_iteration} shows that the efficient 2SG converges much quicker and obtains a better terminal value. One explanation for this phenomenon is that the Hessian information helps the speed of convergence, similar to the benefits of Newton-Raphson relative to the gradient-descent method.

\begin{figure}[!htbp]
	\centering
	\includegraphics[width=\linewidth]{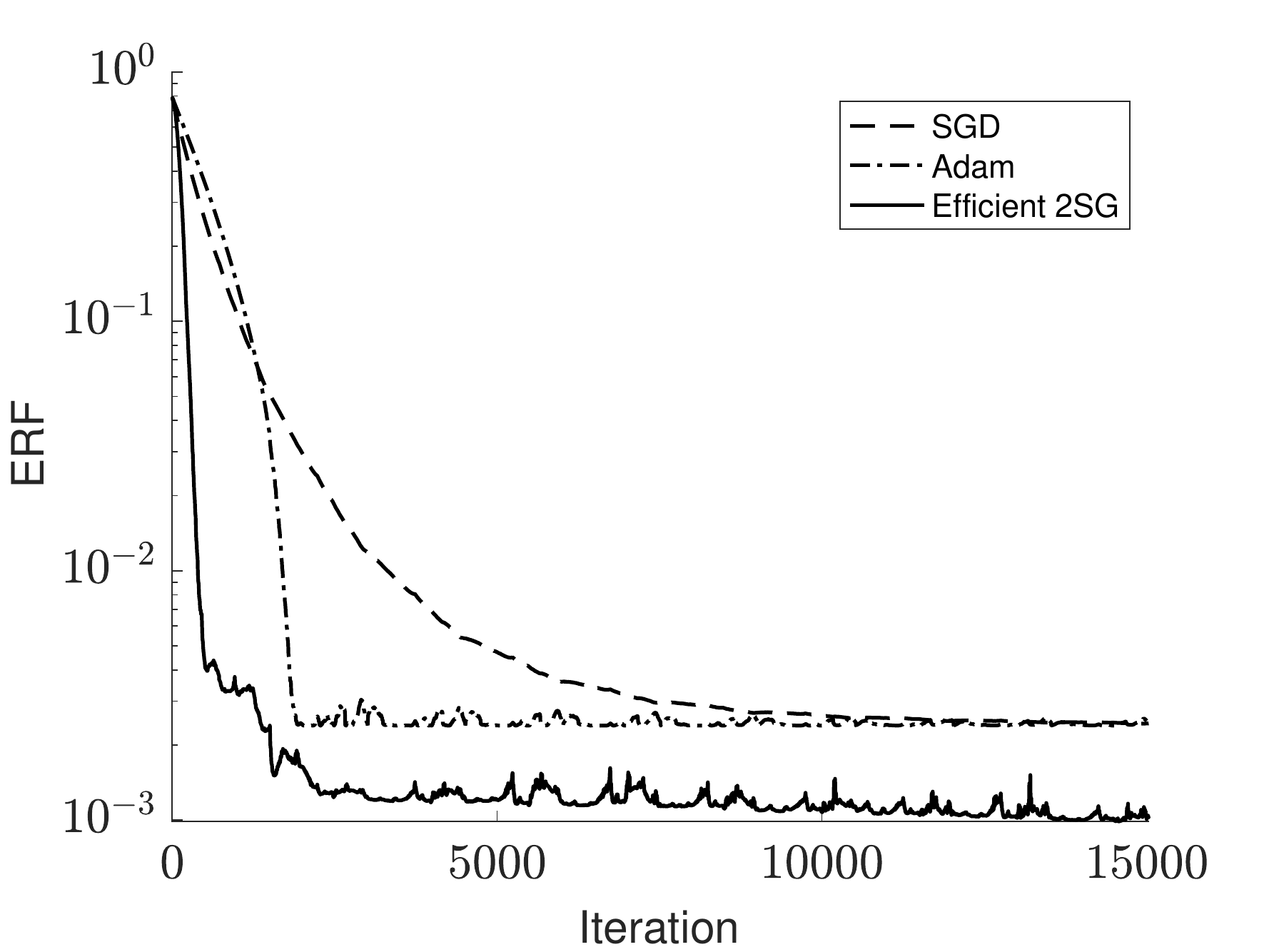}
	\caption{ERF of training samples in SGD, ADAM, and the efficient 2SG under concurrent implementation.}
	\label{fig:real_data_log_MSE_per_iteration}
\end{figure}

Figure~\ref{fig:real_data_log_MSE_per_iteration_adjusted} compares the ERF of the two algorithms in terms of the number of gradient evaluations. Note that each iteration of SGD and ADAM takes one gradient evaluation, while the efficient 2SG takes three gradient evaluations. This comparison is suitable for the non-concurrent implementation since one iteration of the efficient 2SG has roughly the cost of three iterations of the SGD. It is shown in Figure~\ref{fig:real_data_log_MSE_per_iteration_adjusted} that the efficient 2SG still outperforms SGD and ADAM even without any concurrent implementation. There is less than a 7\% difference in running time among SGD, ADAM, and the efficient 2SG under the concurrent implementation.

\begin{figure}[!htbp]
	\centering
	\includegraphics[width=\linewidth]{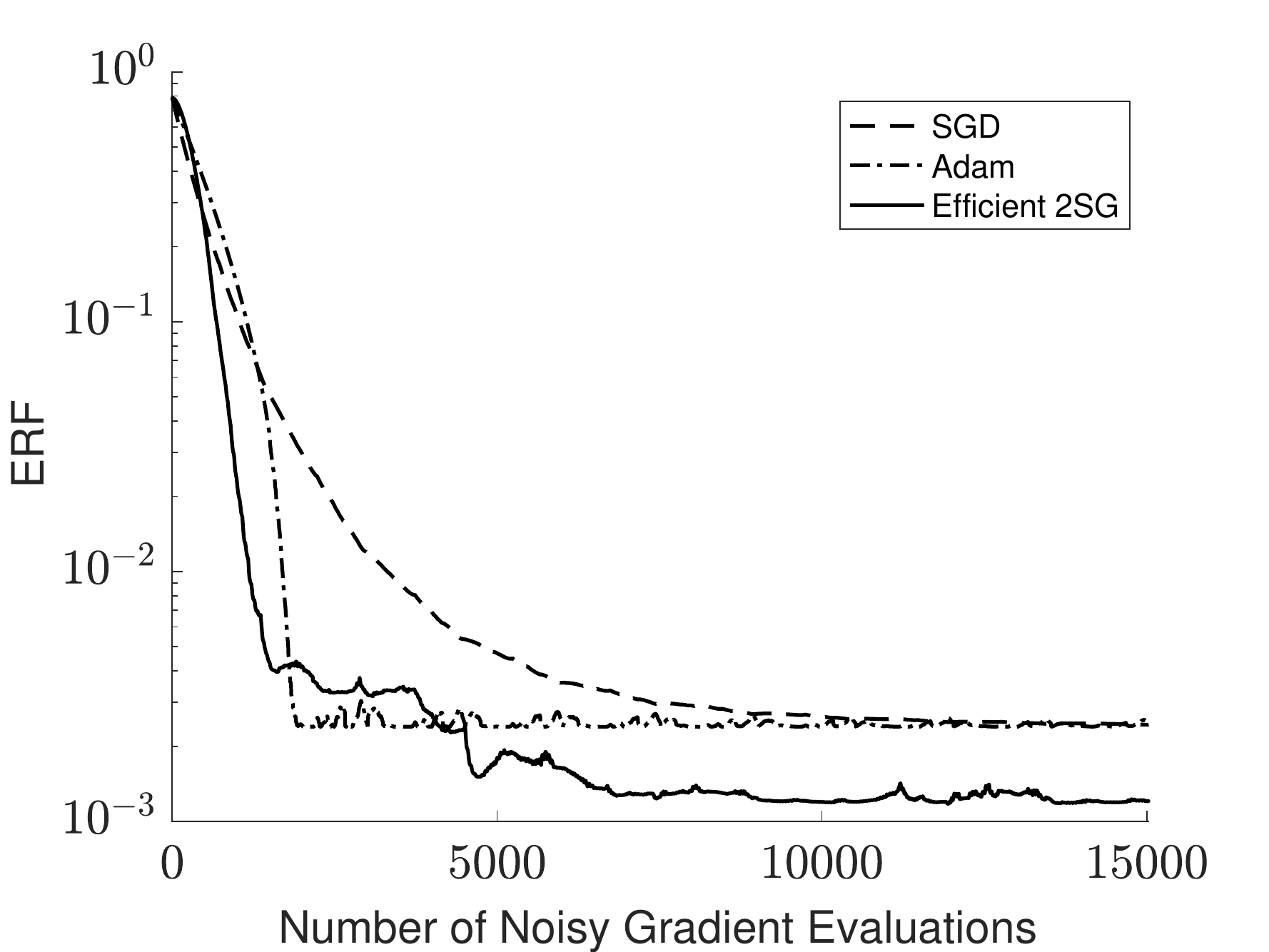}
	\caption{ERF of training samples in SGD, ADAM, and the efficient 2SG per gradient evaluation under \emph{serial} (non-concurrent) computing. SGD and ADAM have 3 times the number of iterations of 2SG.}
	\label{fig:real_data_log_MSE_per_iteration_adjusted}
\end{figure}

\section{Conclusions} \label{sec:conclusion}
To the best of our knowledge, 2SPSA, 2SG, E2SPSA and E2SG are the fastest possible second-order stochastic Newton-type algorithms based on the estimation of the Hessian matrix from either noisy loss measurements or noisy gradient measurements. The algorithms use only a small number of measurements, independent of $p$, at each iteration. This paper shows how symmetric indefinite matrix factorization may be used to reduce the per-iteration FLOPs of the algorithms from $ O(p^3) $ to $ O(p^2) $. The approach guarantees a positive definite estimation of the Hessian matrix (``preconditioned") and a valid stochastic Newton-type update of the parameter vector, both in $ O(p^2) $. This implementation scheme serves to improve practical performance in high-dimensional problems, such as deep learning. In our proposed scheme, the formal convergence and convergence rate for $\hbtheta_k$ and $\obH_k$ are maintained, following the prior work \cite{spall2000adaptive,spall2009feedback}. 

Besides the theoretical guarantee, numerical studies show that the efficient implementation of second-order SP methods provides a promising convergence rate at a tolerable computing cost, compared with stochastic gradient descent method. Note that second-order methods do not provide global convergence in general, and therefore the second-order method is recommended to be implemented after reaching the vicinity of the optimizer.

Overall, our proposed scheme of second-order SA methods has value in high-dimensional optimization and learning problems. Because a key step of this work is the symmetric indefinite factorization, the proposed algorithm might be useful for other algorithms whenever updating an estimated Hessian matrix is involved, such as second-order random directions stochastic approximation \cite{prashanth2017adaptive}, natural gradient descent \cite{amari2000adaptive}, and stochastic variants of the BFGS quasi-Newton methods \cite{schraudolph2007stochastic}. In all those methods, instead of directly updating the matrix of interest (usually the Hessian matrix), one might consider updating its corresponding symmetric indefinite factorization in the manner of this paper to speed up any matrix inverse operation or matrix eigenvalue modification. Overall, the proposed approach provides a practical second-order method that can be used following first-order or other methods that are able to put the iterate in at least the vicinity of the solution.

%
%
%
%




\bibliographystyle{IEEEtran}
\bibliography{Fa2SPSA_reference}

\end{document}